\documentclass[sn-basic,Numbered]{sn-jnl}% Math and Physical Sciences Reference Style
\usepackage{graphicx}%
\usepackage{multirow}%
\usepackage{amsmath,amssymb,amsfonts}%
\usepackage{amsthm}%
\usepackage{mathrsfs}%
\usepackage[title]{appendix}%
\usepackage{xcolor}%
\usepackage{textcomp}%
\usepackage{manyfoot}%
\usepackage{booktabs}%
\usepackage{algorithm}%
\usepackage{algorithmicx}%
\usepackage{algpseudocode}%
\usepackage{listings}%
\usepackage{lineno,hyperref}
\usepackage{amscd}
\usepackage{bm}
\usepackage{color}
\usepackage{cleveref}
\usepackage{epstopdf}
\usepackage{float}
\usepackage{lmodern}
\usepackage{subcaption}
\usepackage{todonotes}
\newtheorem{theorem}{Theorem}%  meant for continuous numbers
\newtheorem{proposition}{Proposition}
\newtheorem{lemma}{Lemma}

\newtheorem{example}{Example}%
\newtheorem{remark}{Remark}%
\newtheorem{definition}{Definition}%

\crefname{equation}{}{}
\renewcommand{\cref}[1]{\namecref{#1}\labelcref{#1}}
\Crefname{equation}{Eq.}{Eqs.}
\crefname{lemma}{Lemma~}{Lemmas~}
\crefname{section}{Section~}{Sections~}
\crefname{theorem}{Theorem~}{Theorems~}
\crefname{remark}{Remark~}{Remarks~}
\crefname{figure}{Fig.~}{Figs.~}
\crefname{table}{Table~}{Tables~}

\newcommand{\vecb}[1]{\boldsymbol{#1}}
\newcommand{\normmm}[1]{{\left\vert\kern-0.25ex\left\vert\kern-0.25ex\left\vert #1
   \right\vert\kern-0.25ex\right\vert\kern-0.25ex\right\vert}}

\begin{document}

\title{Analysis of a $\vecb{P}_1\oplus \vecb{RT}_0$ finite element method for linear elasticity with Dirichlet and mixed boundary conditions}

\author[1]{\fnm{Hongpeng} \sur{Li}}\email{lihongpeng\_sd@163.com}
\author[2,1]{\fnm{Xu} \sur{Li}}\email{xulisdu@126.com}
\author*[1]{\fnm{Hongxing} \sur{Rui}}\email{hxrui@sdu.edu.cn}

\affil[1]{\orgdiv{School of Mathematics}, \orgname{Shandong University}, \city{Jinan}, \postcode{250100}, \country{China}}
\affil[2]{\orgdiv{Eastern Institute for Advanced Study}, \orgname{Eastern Institute of Technology}, \city{Ningbo}, \postcode{315200}, \country{China}}

\abstract{In this paper, we investigate a low-order robust numerical method for the linear elasticity problem.
The method is based on a Bernardi--Raugel-like $\vecb{H}(\mathrm{div})$-conforming method
proposed first for the Stokes flows in [Li and Rui, IMA J. Numer. Anal. \textbf{42} (2022) 3711--3734].
Therein the lowest-order $\vecb{H}(\mathrm{div})$-conforming Raviart--Thomas space ($\vecb{RT}_0$) was added to the
classical conforming $\vecb{P}_1\times P_0$ pair to meet the inf-sup condition,
while preserving
the divergence constraint and some important features of conforming methods.
Due to the inf-sup stability of {the} $\vecb{P}_1\oplus \vecb{RT}_0\times P_0$ pair,
a locking-free elasticity discretization {with respect to} {the Lam\'{e} constant $\lambda$} can be naturally obtained.
Moreover, our scheme is gradient-robust for the pure and homogeneous displacement boundary problem,
that is, the discrete $\vecb{H}^1$-norm of the displacement is $\mathcal{O}(\lambda^{-1})$
when the external body force is a gradient field.
We also consider the mixed displacement and stress boundary problem,
whose $\vecb{P}_1\oplus \vecb{RT}_0$ discretization should be carefully
designed due to a consistency error arising from the $\vecb{RT}_0$ part.
We propose both symmetric and nonsymmetric schemes
to approximate the mixed boundary case.
The optimal error estimates are derived for the energy norm and/or $\vecb{L}^2$-norm.
Numerical experiments demonstrate the accuracy and robustness of our schemes.}

\maketitle

\section{Introduction} \label{Section1}
   This paper is concerned with a low-order finite element method for the linear elasticity problem.
   Assume that $\Omega\subset \mathbb{R}^d (d=2,3)$ is a bounded domain
   with {polyhedral} and Lipschitz-continuous boundary $\partial\Omega$.
   The symmetric $d\times d$ stress tensor is defined as
   \begin{align*}
      \vecb{\sigma}(\vecb{u}):=\left(2\mu\vecb{\epsilon}(\vecb{u})+\lambda(\nabla\cdot \vecb{u})\mathbf{I}\right)
      \quad\text{with}\quad
      \vecb{\epsilon}(\vecb{u}):=(\nabla\vecb{u}+\nabla\vecb{u}^\top)/2,
   \end{align*}
   where $\vecb{u}$ is the displacement of the {elastic material} and
   $\mathbf{I}\in \mathbb{R}^{d\times d}$ is the identity matrix.
   $\lambda$ and $\mu$ are two Lam\'{e} parameters which satisfy $0<\lambda_0<\lambda<\infty$ and $0<\mu_1<\mu<\mu_2$.
   Then the linear elasticity problem with homogeneous displacement boundary condition
   is as follows:
   \begin{equation} \label{M1} %\tag{M1}
      \begin{aligned}
         -\nabla\cdot \vecb{\sigma}(\vecb{u}) & =\vecb{f} &  & \text { in } \Omega,        \\
         \vecb{u}                             & =\vecb{0} &  & \text { on }\partial\Omega,
      \end{aligned}
   \end{equation}
   with $\vecb{f}\in [L^2(\Omega)]^d$ being an external body force.
   Using {the} Green formulation we obtain a variational formulation of (\ref{M1}): Find ${\vecb{u}\in\vecb{V}:=[{H}_{0}^{1}(\Omega)]^d}$ such that
   \begin{equation} \label{M1weak}
      2\mu a(\vecb{u},\vecb{v})+\lambda(\nabla\cdot\vecb{u},\nabla\cdot\vecb{v})=(\vecb{f},\vecb{v})\quad \forall \vecb{v}\in \vecb{V},
   \end{equation}
   where {$(\cdot,\cdot)$ denotes the $L^2$ inner products}, $a(\vecb{u},\vecb{v}):=(\vecb{\epsilon}(\vecb{u}{)},\vecb{\epsilon}(\vecb{v}))$ and 
   {${H}_{0}^{1}(\Omega)$ consists of the functions with vanishing trace in $H^1(\Omega)$}.
   % The following well-known Poincar\'{e} inequality holds: there exists a positive constant $C_p$ such that
   % \begin{equation} \label{eqPoin}
   % \Vert\vecb{v}\Vert\leq C_p\Vert\nabla\vecb{v}\Vert \quad \forall \vecb{v}\in\vecb{V}.
   % \end{equation}
   Two kinds of robustness are considered in this contribution: locking-free property and gradient-robustness.
   The former means that the error estimates do not blow up as
   the Lam\'{e} constant $\lambda\rightarrow\infty$, while the latter means
   the dominant gradient fields in
   the governing equation do not lead to spurious displacement. To be more precise,
   if $\vecb{f}$ is a gradient field, it was proven in \cite{fu2021locking} that
   $\|\nabla\vecb{u}\|=O(\lambda^{-1})$ (i.e., as $\lambda\rightarrow \infty$, the true
   solution $\vecb{u}$ should tend to zero). Then a gradient-robust method should preserve
   this property.
   %First the Poisson ratio $0<\nu:=\frac{\lambda}{2(\lambda+\mu)}<\frac{1}{2}$ fulfills $\nu\approx\frac{1}{2}$, i.e., the Lam\'{e} parameter $\lambda\rightarrow\infty$, and the material is nearly incompressible. Secondly, the exterior volume force $\vecb{f}$ is a gradient field, i.e., there exists a potential $\psi$ with $\vecb{f}=\nabla\psi$.

   The locking phenomenon in elasticity problems is usually called ``volume locking" or ``Poisson locking".
   When $\lambda$ is very large, the material is nearly incompressible (i.e., $\nabla\cdot\vecb{u}\approx 0$).
   The standard finite element method, such as the continuous piecewise linear element, can behave very badly \cite{brenner2008mathematical,phillips2009overcoming}.
   Babu\v{s}ka and Suri \cite{babuvska1992locking} found that any polynomial of degree $k\geq 1$ cannot avoid locking on quadrilateral mesh.
   Volume locking has been dealt with in many different discretization approaches. We divide the discretization approaches into three large classes.
   The first class is based on the primal displacement equation (\ref{M1}).
   A variety of finite element methods have been implemented for this, such as {the} nonconforming
   Crouzeix--Raviart (CR) element \cite{HL2003}, {the} enriched Galerkin method \cite{yi2022locking}, {the} 
   weak Galerkin method \cite{yi2019lowest} and {the} discontinuous Galerkin method \cite{wihler2006locking}, to name just a few.
   The second class is to introduce the ``solid pressure" $p=\lambda \nabla \cdot \vecb{u}$ as
   an independent unknown. Then the primal formulation (\ref{M1}) can be reformulated as
   a generalized Stokes problem
   \begin{equation} \label{M1Change}
      -2\mu\nabla\cdot\vecb{\epsilon}(\vecb{u})-\nabla p=\vecb{f}, \quad
      \nabla\cdot\vecb{u}-\lambda^{-1}p=0  \text { in } \Omega.
   \end{equation}
   %When the material is nearly incompressible, i.e., $\lambda\rightarrow\infty$, It becomes an imcompressible Stokes problem.
   Any inf-sup stable mixed element method which is appropriate {for} the Stokes problem,
   would provide a locking-free formulation for the linear elasticity problem,
   cf. \cite{boffi2013mixed,chiumenti2002stabilized,john2016finite,rui2018locking}.
   Note that an inf-sup stable mixed formulation can usually be transformed into a primal
   formulation by static condesation
   of the pressure unknowns
   if the discrete pressure is discontinuous.
   The last class {transforms} the linear elasticity equations (\ref{M1}) into the Hellinger-Reissner formulation
   based on {the} Hellinger-Reissner variational principle \cite{arnold2007mixed}.
   And this method {produces} direct approximations to both stress and displacement.
   The most popular methods include mixed finite element {methods}
   \cite{arnold2002mixed,arnold2007mixed,lamichhane2009inf,hu2015finite,hu2016finite}, dual-mixed {methods}
   \cite{gatica2006analysis,gatica2007dual,gatica2009augmented}, and
   hybrid discontinuous Galerkin {methods} \cite{cockburn2013superconvergent,qiu2018hdg}.

   Compared to volume locking, gradient robustness is a new concept. The definition of gradient-robustness is given in Section \ref{Section2}.
   A related concept of gradient-robustness has been introduced first for the steady compressible isothermal Stokes equations in \cite{akbas2020gradient}.
   For the incompressible Stokes problem, gradient-robustness means pressure-robustness \cite{john2017divergence}, that is,
   when the external force in the momentum equation is a gradient field, it is only balanced by the pressure gradient.
   Fu et al. \cite{fu2021locking} proposed and analyzed an $\vecb{H}(\mathrm{div})$-conforming HDG scheme for linear elasticity (\ref{M1}),
   and the scheme is both locking-free and gradient-robust.
   Basava and Wollner \cite{basava2022pressure} applied 
   the pressure-robust reconstruction methods for the Stokes problem \cite{linke2014role,LMT2016} to 
   elasticity discretizations to get a gradient-robust method.
   Numerical schemes of linear elasticity may perform well when the body force is divergence-free,
   but may fail when the body force is a gradient field.
   The concept of gradient robustness gives us a new perspective to analyze the effectiveness of numerical schemes.
   Our goal is to construct algorithms that maintain parameter robustness about $\lambda$,
   and are also accurate when the body force in the momentum balance equation is dominated by a gradient field.

   We also consider mixed boundary conditions in this paper.
   Assume that the boundary $\partial\Omega$ consists of two parts:
   $\Gamma_D\subset\partial\Omega$, with $\vert\Gamma_D\vert>0$, and $\Gamma_N:=\partial\Omega\setminus\Gamma_D$.
   The elasticity problem with mixed boundary conditions becomes the primal formulation in (\ref{M1}) with
   \begin{equation} \label{M2}%\tag{M2}
      \vecb{u}=\vecb{0} \,\text{ on }\Gamma_D,\quad\vecb{\sigma}\vecb{n}=\vecb{g} \,\text{ on }\Gamma_N.
   \end{equation}
   We define $\vecb{V}_{\Gamma_D}:=\{\vecb{v}\in [H^1(\Omega)]^d:\vecb{v}|_{\Gamma_D}=\vecb{0}\}$ {and the} traction $\vecb{g}\in [H_{00}^{1/2}(\Gamma_N)]^d$,
   where $[H_{00}^{1/2}(\Gamma_N)]^d:=\left\{\vecb{v}|_{\Gamma_N}:\vecb{v}\in\vecb{V}_{\Gamma_D}\right\}$.
   The associated duality pairing with respect to the $[L^2(\Gamma_N)]^d$ is denoted by $\langle\cdot,\cdot\rangle_{\Gamma_N}$.
   The variational formulation of (\ref{M2}) is that: Find $\vecb{u}\in\vecb{V}_{\Gamma_D}$ such that
   \begin{equation} \label{M2weak}
      2\mu a(\vecb{u},\vecb{v})+\lambda(\nabla\cdot\vecb{u},\nabla\cdot\vecb{v})=(\vecb{f},\vecb{v})+\langle\vecb{g},\vecb{v}\rangle_{\Gamma_N}\quad \forall \vecb{v}\in \vecb{V}_{\Gamma_D}.
   \end{equation}
   Under the assumption $\vert\Gamma_D\vert>0$, the Korn's inequality holds {\cite{brenner2008mathematical}}, i.e., there exists a positive constant such that
   \begin{equation} \label{eqKorn}
      \Vert\vecb{v}\Vert_1 \leq C_{korn}\Vert\vecb{\epsilon}(\vecb{u})\Vert \quad \forall \vecb{u}\in \vecb{V}_{\Gamma_D},
   \end{equation}
   where {$\Vert\cdot\Vert_1$ and $\Vert\cdot\Vert$ denote the usual $H^1$ norms and $L^2$ norms, respectively}. Thus the unique solvability of (\ref{M1weak}) and (\ref{M2weak}) holds.

   The starting point of this paper is a kind of low-order inf-sup stable mixed element method for the Stokes
   problem proposed in \cite{li2022low}, where
   the velocity space is obtained by enriching the space of conforming piecewise linear polynomials ($\vecb{P}_1$) with the
   $\vecb{H}(\mathrm{div})$-conforming
   lowest-order Raviart--Thomas space, and the pressure space consists of piecewise constants with zero mean ($P_0$).
   The resulting scheme is divergence-free and pressure-robust, and the discrete formulation consists of volume integrals only,
   which is different from the usual $\vecb{H}(\mathrm{div})$-conforming methods with discontinuous Galerkin (DG) formulation
   \cite{Wang2007,CKS2007,john2017divergence,fu2021locking}.
   Moreover, since the pressure is discontinuous, as it is mentioned before, this method naturally leads to a locking-free
   primal discretization of the elasticity problem. The main contribution of this paper is twofold: on the one hand, for the
   pure Dirichlet problem, we prove that the resulting method is gradient-robust, on the other hand,
   for the mixed boundary problem, we point out that the extension from pure Dirichlet boundary case
   is not straight-forward and then
   propose some strategies to deal with the Neumann part, which is not involved in \cite{li2022low}.
   {A priori error estimate shows} that all schemes achieve optimal convergence order in the energy norm and $\vecb{L}^2$-norm.
   Numerical experiments show that
   an inappropriate treatment to mixed boundary conditions can lead to
   reduction of the convergence rates of the discrete solution. And our proposed $\vecb{P}_1\oplus \vecb{RT}_0$ schemes are
   numerically accurate and robust for both Dirichlet and mixed boundary conditions. 
   For mixed boundary problems,
   in contrast {to} the Crouzeix--Raviart element method using interior jump stabilization \cite{HL2003},  
   our formulation consists of standard volume integrals and face integrals over the Neumann boundary. 
   Note that 
   interior face integrals like jump stabilization can dramatically change the sparsity of  
   the coefficient matrix. Although this issue of Crouzeix--Raviart element can be bypassed 
   by replacing one component with conforming linear elements \cite{KS1995}, its three-dimensional extension
   has to apply higher-order elements for one component \cite{ZZ2017,HS2018}.

   % In this paper, we propose several discretization schemes to approximate linear elasticity problem with homogeneous displacement boundary condition or 
   % mixed boundary condition. These schemes are based the $\vecb{H}(\mathrm{div})$-conforming $\vecb{P}_1\oplus \vecb{RT}_0$ finite element, 
   % which is proposed first for Stokes flows \cite{li2022low}. This element is the continuous vector-valued piecewise linear polynomial 
   % space $(\vecb{P}_1)$ enriched by the lowest-order Raviart-Thomas space ($\vecb{RT}_0$).
   % We prove the unique solvability of all schemes. 
   % Our schemes are locking-free, i.e., parameter-robust in terms of the Lam\'{e} constant $\lambda$.
   % The priori error estimates show that all schemes achieve optimal convergence order in the energy norm and $L^2$ norm.
   % Moreover when we consider the homogeneous displacement boundary condition, our $P1\oplus RT0$ scheme is not only locking-free but also gradient-robust on unstructured mesh grid.
   %��������p����P1��RT0�ǵȼ۵ġ�
   %Surprisingly we find that the $P1\oplus RT0$ finite element method based on the primal formulation (\ref{M1}) is equivalent to the $\vecb{P}_1\oplus \vecb{RT}_0$ method when applied to the mixed formulation (\ref{M1Change}).
   %����Ӧ����ɢ�ȵ�������

   The rest of the paper is organized as follows.
   In Section \ref{Section2}, we present the fundamental results about gradient-robustness. We discuss some finite element schemes from the perspective of locking-free and gradient-robust properties.
   In Section \ref{Section3} we propose the $\vecb{P}_1\oplus \vecb{RT}_0$ finite element schemes and analyze the well-posedness.
   The uniform convergence analysis about $\lambda$ and the gradient-robustness are analyzed in Section \ref{Section4}.
   The case of mixed boundary conditions is considered in Section \ref{Section5}.
   Finally we do some numerical studies in Section \ref{Section6}.

   % {Throughout the paper we use $c$ or $C$, with or without a subscript, to denote a generic constant.
   %    Let $D\subset\mathbb{R}^d$ be a bounded domain.
   % 	% then the space $H^s(\Omega)$ for integer $s$ is $H^s(E)=\{v\in L^2(E):D^{\omega}v\in L^2(E), \forall 0\leq|\omega|\leq s\}$, 
   % 	% where $L^2(E)$ is the space of square-integrable functions, $D^{\omega}v$ is the distributional derivative. 
   % 	$H^s(D), s\in \mathbb{Z}$, denotes the standard Sobolev space \cite{brenner2008mathematical} and 
   % 	Denote by $\vert\cdot\vert_{s,D}$ and $\Vert\cdot\Vert_{s,D}$, respectively, the Sobolev seminorm and norm in $H^s(D)$. 
   % 	When $s=0$, $H^s(D)$ coincides with $L^2(D)$.
   % 	The inner product and norm in $L^2(D)$ are denoted by $(\cdot,\cdot)_D$ and $\Vert\cdot\Vert_D$, respectively. If $D=\Omega$ the subscript $D$ will be dropped.
   % 	}

   \section{Gradient-robustness and some finite element schemes} \label{Section2}
   In this section and next two sections, the elasticity problem with homogeneous displacement boundary conditions is considered. 
   With convention the boundary value problem with pure displacement (resp. traction) boundary conditions is called a pure displacement (resp. traction) problem. 
   We introduce some fundamental results about Helmholtz decomposition and Helmholtz projector \cite{linke2014role}.
   {Let us define
   \begin{align*}
      \vecb{H}(\mathrm{div};\Omega)&:=\{\vecb{v}\in [L^2(\Omega)]^d: \nabla\cdot\vecb{v}\in L^2(\Omega)\}.
   \end{align*}}
   Every vector field $\vecb{f}\in [L^2(\Omega)]^d$ has a unique decomposition {into a irrotational field $\nabla\phi$ with $\phi\in H^1(\Omega)/\mathbb{R}$ and a divergence-free component $\vecb{f}_0\in\vecb{H}(\mathrm{div};\Omega)$}, i.e., $\vecb{f}=\nabla\phi+\vecb{f}_0$.
   {Moreover}, $\vecb{f}_0$ is $\vecb{L}^2$-orthogonal to $\nabla\psi$ for all $\psi\in H^1(\Omega)$.
   %$\mathbb{P}(\vecb{f}):=\vecb{f}_0$ is the Helmholtz projector, which preserves the divergence-free part.
   {The Helmholtz projector $\mathbb{P}(\vecb{f}):=\vecb{f}_0$ preserves the divergence-free part.}
   For any $\psi\in H^1(\Omega)$, {it holds} $\mathbb{P}(\nabla\psi)=0$.
   We define the divergence-free subspace $\vecb{V}_{\mathrm{div}}$ and its orthogonal complement $\vecb{V}^{\perp}_{\mathrm{div}}$ by
   $$
      \vecb{V}_{\mathrm{div}}:=\left\{ \vecb{u}\in \vecb{V}:(\nabla\cdot\vecb{u},q)=0 \quad\forall q\in W \right\}
      =\left\{ \vecb{u}\in \vecb{V}:\nabla\cdot\vecb{u}=0 \right\},
   $$
   $$
      \vecb{V}^{\perp}_{\mathrm{div}}:=\left\{ \vecb{u}\in \vecb{V}:a(\vecb{u},\vecb{v})=0 \quad\forall\vecb{v}\in\vecb{V}_{\mathrm{div}} \right\},
   $$
   respectively, where $W:=L^{2}_0(\Omega)=\{q\in L^2(\Omega):\int_\Omega q\,d\vecb{x}=0\}$ is the space of $L^2$ functions with zero mean.
   Now any function $\vecb{u}\in\vecb{V}$ can be uniquely decomposed as $\vecb{u}=\vecb{u}^0+\vecb{u}^{\perp}\in\vecb{V}_{\mathrm{div}}\oplus\vecb{V}^{\perp}_{\mathrm{div}}$.
   The following boundedness of $\vecb{u}$ in {the} $\vecb{H}^1$-norm can be derived.
   \begin{lemma} \label{lemStab}
      Let $\vecb{u}$ be the solution of (\ref{M1weak}). It satisfies the following stability estimate:
      \begin{equation} \label{eqlem3}
         \Vert\vecb{u}\Vert_1 \leq {\frac{1}{2\mu/C_{korn}^2+\lambda\beta^2}}\Vert\vecb{f}\Vert_{-1} + {\frac{C_{korn}^2}{2\mu}}\Vert\mathbb{P}(\vecb{f})\Vert_{-1},
      \end{equation}
      where $\|\cdot\|_{-1}$ denotes the $\vecb{H}^{-1}$ norm.
   \end{lemma}
   \begin{proof}
      The inf-sup condition \cite{john2016finite}
      \begin{align}\label{continuousinfsup}
         \sup_{\vecb{v}\in\vecb{V}\setminus\{\vecb{0}\}}\frac{(\nabla\cdot\vecb{v},q)}{\Vert\vecb{v}\Vert_1}\geq \beta\Vert q\Vert
      \end{align}
      implies the divergence operator is bijective from $\vecb{V}^{\perp}_{\mathrm{div}}$ to $W$, and
      \begin{equation} \label{eqDivMap}
         \Vert\vecb{u}^{\perp}\Vert_1\leq\frac{1}{\beta}\Vert\nabla\cdot\vecb{u}^{\perp}\Vert.
      \end{equation}
      We set $\vecb{v}=\vecb{u}^{\perp}$ in the variational formulation (\ref{M1weak}), and use the $a(\cdot,\cdot)$-orthogonality to get
      \begin{align*}
         2\mu a(\vecb{u}^{\perp},\vecb{u}^{\perp})+\lambda(\nabla\cdot\vecb{u}^{\perp},\nabla\cdot\vecb{u}^{\perp})=(\vecb{f},\vecb{u}^{\perp}).
      \end{align*}
      By Korn's inequality (\ref{eqKorn}), (\ref{eqDivMap}) and the Cauchy-Schwarz inequality, we get
      \begin{equation} \label{eqUperp2}
         \frac{2\mu}{C_{korn}^2}\Vert\vecb{u}^{\perp}\Vert^2_1 + \lambda\beta^2\Vert\vecb{u}^{\perp} \Vert^2_1\leq \Vert\vecb{f}\Vert_{-1}\Vert\vecb{u}^{\perp}\Vert_1,
      \end{equation}
      Thus
      \begin{equation} \label{eqUperp}
         \Vert\vecb{u}^{\perp}\Vert_1 \leq {\frac{1}{2\mu/C_{korn}^2+\lambda\beta^2}}{\Vert\vecb{f}\Vert_{-1}}.
      \end{equation}
      Next, testing with an arbitrary divergence-free function $\vecb{v}^0\in\vecb{V}_{\mathrm{div}}$ in (\ref{M1weak}) gives
      \begin{align*}
         2\mu a(\vecb{u}^0,\vecb{v}^0)=(\vecb{f},\vecb{v}^0)=(\mathbb{P}(\vecb{f}),\vecb{v}^0).
      \end{align*}
      Setting $\vecb{v}^0=\vecb{u}^0$, by Korn's inequality (\ref{eqKorn}) it yields
      \begin{align*}
         \frac{2\mu}{C_{korn}^2}\Vert\vecb{u}^0\Vert^2_1  \leq \Vert\mathbb{P}(\vecb{f})\Vert_{-1}\Vert\vecb{u}^0\Vert_1,
      \end{align*}
      which implies
      \begin{equation} \label{eqU0}
         \Vert\vecb{u}^0\Vert_1 \leq {\frac{C_{korn}^2}{2\mu}}\Vert\mathbb{P}(\vecb{f})\Vert_{-1}.
      \end{equation}
      Then inequality (\ref{eqlem3}) follows immediately from (\ref{eqUperp}) and (\ref{eqU0}).
   \end{proof}
   Following the proof of Lemma \ref{lemStab}, we can easily prove the following lemma, which is introduced in \cite{fu2021locking}.
   It characterizes an important feature of the exact solution of nearly incompressible linear elasticity.

   \begin{lemma}
      If $\vecb{f}$ in (\ref{M1}) is a gradient field,
      i.e., $\vecb{f}=\nabla\phi$, $\phi\in H^1(\Omega)/\mathbb{R}$,
      the solution $\vecb{u}=\vecb{u}^0+\vecb{u}^{\perp}$ of (\ref{M1weak}) satisfies,
      \begin{equation} \label{eqTh1}
         \vecb{u}^0=\vecb{0},\quad \Vert\vecb{u}\Vert_1=\Vert\vecb{u}^{\perp}\Vert_1\leq{\frac{1}{2\mu/C_{korn}^2+\lambda\beta^2}}\Vert\phi\Vert.
      \end{equation}
   \end{lemma}
   \begin{proof}
      The proof of Lemma \ref{lemStab} implies $\vecb{u}^0=\vecb{0}$. For the inequality in \cref{eqTh1}, one has
      \begin{align*}
         \frac{2\mu}{C_{korn}^2}\Vert\vecb{u}^{\perp}\Vert^2_1 + \lambda\beta^2\Vert\vecb{u}^{\perp}\Vert^2_1 \leq (\nabla\phi,\vecb{u}^{\perp}) \leq (\phi, \nabla\cdot\vecb{u}^{\perp})\leq \Vert\phi\Vert \Vert\vecb{u}^{\perp}\Vert_1.
      \end{align*}
      Then the inequality can be obtained analogously to \cref{eqUperp}.
   \end{proof}

   \begin{definition}[Gradient-robustness] \label{gradRobDef}
      A discretization of the linear elasticity problem (\ref{M1}) is called gradient-robust, if on an arbitrary but fixed grid, the discrete displacement solution $\vecb{u}_h$ satisfies
      \begin{align*}
         \Vert\vecb{u}_h\Vert_{1,h}=\mathcal{O}(\lambda^{-1})
      \end{align*}
      in case $\vecb{f}$ is a gradient field, where $\Vert\cdot\Vert_{1,h}$ is a discrete $H^1$ norm defined in finite element spaces.
   \end{definition}
   It is worth noting that in this paper the gradient-robustness is only considered under homogeneous displacement boundary conditions.
   In \cite{zdunek2023pressure} the role of some different boundary conditions on pressure-robustness for the incompressible linear elasticity problem is discussed,
   such as the normal or tangential components of displacement boundary conditions.

   According to \cref{M1Change}, when $\lambda\rightarrow\infty$,
   the elasticity problem tends to a Stokes problem
   \begin{equation} \label{EqStokes}
      -2\nu\nabla\cdot\vecb{\epsilon}(\vecb{u})-\nabla p=\vecb{f}, \quad\nabla\cdot\vecb{u}=0  \text { in } \Omega.
   \end{equation}
   %дһдpressure-robust �� gradient-robust �����⣬�Ƶ������ڵڶ��ڣ��������
   Consider the case $\vecb{f}=\nabla\phi$ with $\phi\in H^1(\Omega)/\mathbb{R}$.
   For \cref{EqStokes} it holds $(\vecb{u},p)=(\vecb{0},-\phi)$ (cf. \cite{john2017divergence}),
   while for \cref{M1} (or \cref{M1Change})
   the displacement solution of a gradient-robust method tends to zero from Definition~\ref{gradRobDef}
   when $\lambda\rightarrow\infty$, which coincides with
   the solution of its Stokes limit. In this sense, we say a gradient-robust discretization
   is asymptotic preserving (AP) \cite{jin1999efficient}.
   From \cite{john2017divergence}
   the discrete velocity solution of a pressure-robust method for \cref{EqStokes}
   is also zero. An elasticity discretization should be gradient-robust if it corresponds to
   a pressure-robust scheme for the Stokes problem \cite{fu2021locking}.

   Next we list some finite element schemes, and take some experiments to verify their properties about {locking-free} and gradient-robustness.
   We omit the details of these elements, which can be found in other references.
   $P_k(T)\, (k\geq 0)$ denotes the space of polynomials of degree no more than $k$ on an element $T$.
   We assume the region is unit square, i.e., $\Omega=(0,1)^2$.
   We use the uniform triangular partition $\mathcal{T}_h$ {(see Fig. \ref{meshGrid} (left))}, where the spatial steps $h$ range from $1/8$ to $1/128$.
   Example \ref{example1} is designed to satisfy $\nabla\cdot\vecb{u}\rightarrow 0$ when the Lam\'{e} constant $\lambda\rightarrow\infty$. 
   It is taken to verify locking-free property. The right-hand term $\vecb{f}$ is determined by Equation (\ref{M1}).
   Example \ref{example2} is taken from \cite{fu2021locking}, the right-hand term is designed to be a gradient field to verify {the gradient-robustness} property.
   We use homogeneous Dirichlet boundary conditions for both examples. The Lam\'{e} constants {equal} $\lambda=1,10^2,10^4,10^6$, and $\mu=1$.
   \begin{example} \label{example1}
      The exact solutions are chosen as follows
      \begin{flalign*}
         \renewcommand{\arraystretch}{1.5}
         \left\{
         \begin{array}{l}
            u_1=sin(2\pi y)(-1+cos(2\pi x))+\frac{1}{\mu+\lambda}sin(\pi x)sin(\pi y), \\
            u_2=sin(2\pi x)(1-cos(2\pi y))+\frac{1}{\mu+\lambda}sin(\pi x)sin(\pi y).
         \end{array}\right.
      \end{flalign*}
   \end{example}
   \begin{example} \label{example2}
      We take $\vecb{f}=\nabla\psi$ with $\psi=x^6+y^6$.
   \end{example}

   %P1
   \noindent\textbf{$\vecb{P}_1$ {scheme}.}
   Let $\vecb{V}_{h}=\{\vecb{v}\in\vecb{V}:\left.\vecb{v}\right|_{T}\in\left[P_{1}(T)\right]^{2} \,\forall T \in\mathcal{T}_{h}\}$ be the piecewise linear continuous finite element space. The finite element scheme for (\ref{M1}) is that: Find $\vecb{u}_h\in\vecb{V}_h$ such that
   \begin{equation} \label{SchemeP1}%\tag{S1}
      2\mu a(\vecb{u}_{h}, \vecb{v}_{h})+\lambda(\nabla\cdot\vecb{u}_{h},\nabla\cdot\vecb{v}_{h})=(\vecb{f}, \vecb{v}_{h}) \quad\forall\vecb{v}_{h}\in\vecb{V}_{h}.
   \end{equation}
   It is well-known that the continuous piecewise linear element would result in {a} poor convergence rate of the displacement.
   We take $W_h=\{q\in W:\left.q\right|_{T}\in P_0(T)\,\forall T\in\mathcal{T}_{h}\}$ such that $\nabla\cdot\vecb{V}_h\subset W_h$. 
   As shown in Figure \ref{figP1} (left), when $\lambda=1$ we obtain the optimal convergence rate of the displacement. {As $\lambda$ becomes large, the convergence rate deteriorates on the chosen meshes.}
   Figure \ref{figP1} (right) shows that $\Vert\nabla\vecb{u}_h\Vert=\mathcal{O}(\lambda^{-1})$. The scheme (\ref{SchemeP1}) is gradient-robust, but it is not free of volumetric locking.
   \begin{figure}[htbp]
      \centering
      \includegraphics[scale=0.5]{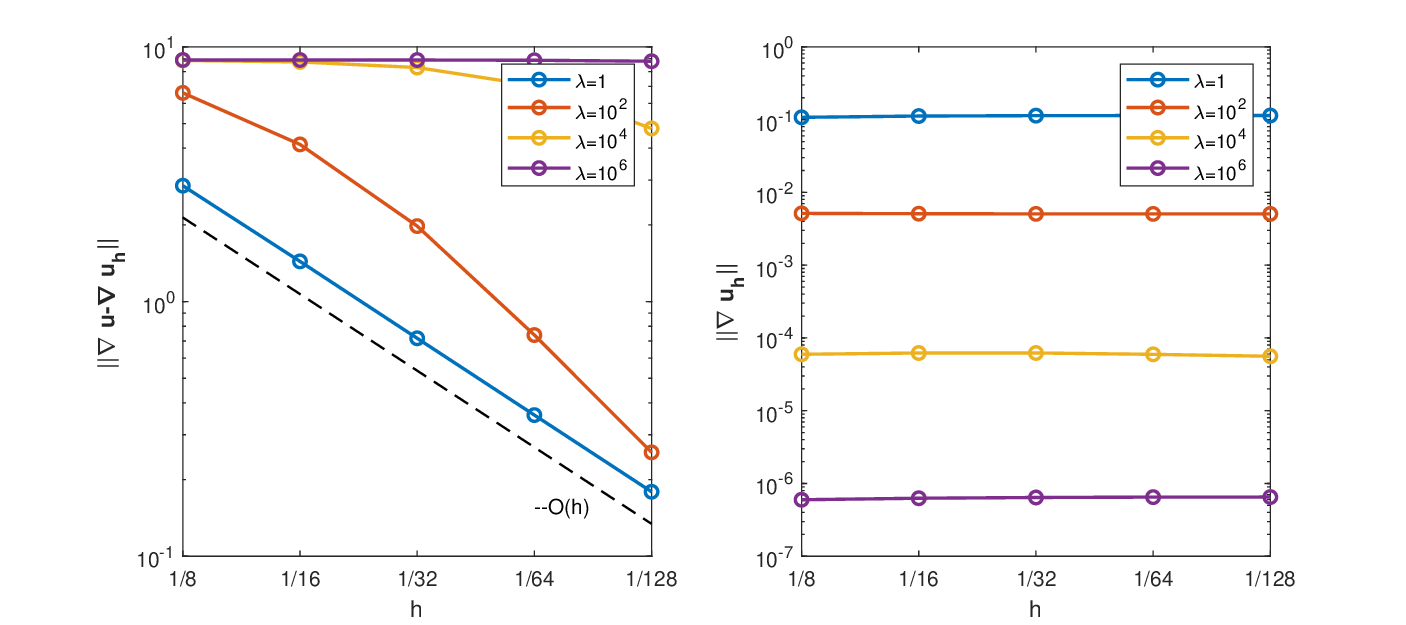}
      \caption{$\vecb{P}_1$ element for Example \ref{example1} (left) and Example \ref{example2} (right).}
      \label{figP1}
   \end{figure}

   %BR
   \noindent\textbf{BR {scheme}.}
   The scheme is motivated by the discretization for the poroelasticity problem \cite{yi2017study}. The Bernardi--Raugel (BR) pair \cite{bernardi1985analysis} is the linear space enriched by edge bubble functions. We choose $\vecb{V}_h\subset\vecb{V}$ to be the BR element space.
   The finite element scheme for (\ref{M1}) is that: Find $\vecb{u}_h\in\vecb{V}_h$ such that
   \begin{equation} \label{SchemeBR}%\tag{S2}
      2\mu a(\vecb{u}_{h}, \vecb{v}_{h})+\lambda(P_h\nabla\cdot\vecb{u}_{h},P_h\nabla\cdot\vecb{v}_{h})=(\vecb{f}, \vecb{v}_{h}) \quad\forall\vecb{v}_{h}\in\vecb{V}_{h}.
   \end{equation}
   $P_h$ is the orthogonal $L^2$ projection defined in Section \ref{Section3}. In (\ref{SchemeBR}) we implement the technique of reduced integration \cite{malkus1978mixed} to obtain the uniform convergence {with respect to} $\lambda$.
   In the poroelasticity problem, the linear elasticity equation is used to describe the displacement of the solid medium. The proof of the uniform convergence of (\ref{SchemeBR}) with respect to $\lambda$ can be derived from \cite{yi2017study}.
   As shown in Figure \ref{figBR} (left), the lines for $\lambda=1,10^2,10^4,10^6$ are coincident. The difference {for} different $\lambda$ is {very small}. So the scheme (\ref{SchemeBR}) is uniformly convergent about $\lambda$.
   However, the Bernardi--Raugel element includes the edge bubble functions, which is piecewise quadratic polynomials. It is well-known that the classical Stokes discretization with this pair
   is not divergence-free or pressure-robust \cite{john2017divergence}, and hence \cref{SchemeBR} is not gradient-robust.

   {One should also note that, although a classical discretization with the Bernardi--Raugel element is not gradient-robust, it has been proven in \cite{basava2022pressure} that
   the lack of gradient-robustness can be overcome by a reconstruction strategy from the Stokes discretizations such as  
   \cite{LMT2016,LM:2016}, which is the so-called pressure-robust reconstruction.} For the Crouzeix--Raviart element method below,
   such a strategy is also considered and employed.
   \begin{figure}[htbp]
      \centering
      \includegraphics[scale=0.5]{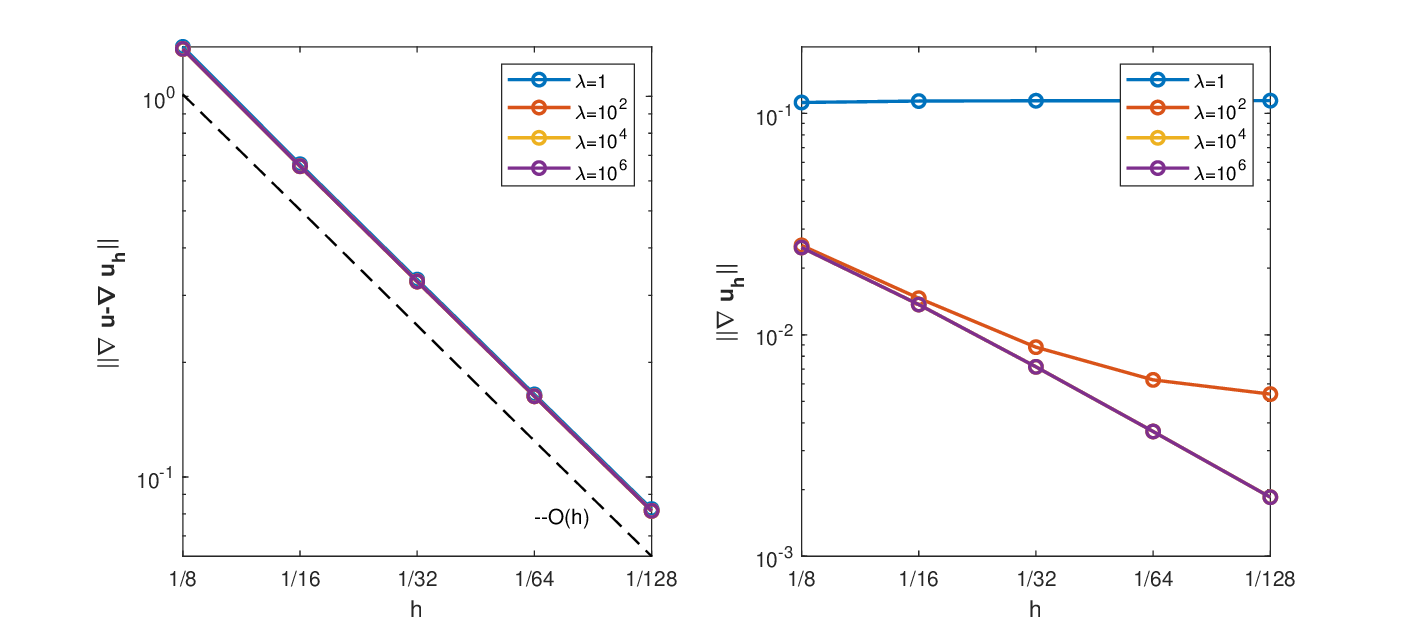}
      \caption{BR element for Example \ref{example1} (left) and Example \ref{example2} (right).}
      \label{figBR}
   \end{figure}

   \noindent\textbf{CR {scheme}.}
   Inspired by the reconstruction method \cite{linke2014role,basava2022pressure}, 
   we try the following gradient-robust reconstruction scheme.
   We use the first order nonconforming Crouzeix--Raviart (CR) element \cite{crouzeix1973conforming} to approximate the linear elasticity (\ref{M1}). When the boundary condition is homogeneous displacement condition, the equation (\ref{M1}) can be rewritten as
   \begin{equation} \label{eqCRChange}
      -\mu\Delta\vecb{u}-(\mu+\lambda)\nabla(\nabla\cdot\vecb{u})=\vecb{f}.
   \end{equation}
   Let $\vecb{V}_h$ be the CR element space. Its locking-free property is derived by Brenner {and Sung} \cite{brenner1992linear}. The piecewise gradient and piecewise divergence operators are defined as
   $$
      (\nabla_h\vecb{u},\nabla_h\vecb{v}):=\sum_{T\in\mathcal{T}_h}(\nabla\vecb{u},\nabla\vecb{v})_T, \quad
      (\nabla_h\cdot\vecb{u},\nabla_h\cdot\vecb{v}):=\sum_{T\in\mathcal{T}_h}(\nabla\cdot\vecb{u},\nabla\cdot\vecb{v})_T,
   $$
   respectively.
   To obtain the gradient-robustness property, we apply a reconstruction operator to the test function. Then the finite element scheme is that: Find $\vecb{u}_h\in\vecb{V}_h$ such that
   \begin{equation} \label{SchemeCR}%\tag{S3}
      \mu(\nabla_h\vecb{u}_h, \nabla_h\vecb{v}_h) + (\mu+\lambda)(\nabla_h\cdot\vecb{u}_h, \nabla_h\cdot\vecb{v}_h) = (\vecb{f},\Pi_h^{R}\vecb{v}_h) \quad\forall\vecb{v}_{h}\in\vecb{V}_{h}.
   \end{equation}
   Here $\Pi_h^{R}: \vecb{V}_h\rightarrow \vecb{V}_h^R$ is the lowest order Raviart--Thomas interpolation defined by \cref{eq:pihr}, where
   $\vecb{V}_h^R$ is the lowest-order Raviart--Thomas space (see \cref{def:pihr}). This reconstruction method is first introduced in \cite{linke2014role}.
   As shown in Figure \ref{figCR}, the scheme (\ref{SchemeCR}) is both locking-free and gradient-robust. Note that the variational form from (\ref{eqCRChange}) is only valid
   for pure displacement problems. If $\Gamma_N\neq \emptyset$, a variational form from (\ref{M1}) can be employed. {In this case} 
   a stabilized version is essential to guarantee the discrete Korn's inequality for CR elements \cite{HL2003}, where
   the jump stabilization need to be added to {interior faces}. 
   
   % Note that Equation (\ref{eqCRChange}) and Equation (\ref{M1}) differ in mixed or Neumann boundary conditions. 
   % If we use the primal formulation (\ref{M1}), a stabilized version is essential to guarantee the discrete Korn's inequality \cite{HL2003}, 
   % and the jump stabilization need to be added to {interior faces under mixed boundary conditions}.

   \begin{proposition}
      The finite element scheme (\ref{SchemeCR}) is gradient-robust in the sense of Definition \ref{gradRobDef}.
   \end{proposition}
   \begin{proof}
      Like in the continuous case, we define the discretely divergence-free space $\vecb{V}_h^0$ and its orthogonal complement $\vecb{V}_h^{\perp}$ as
      $$
         \vecb{V}_h^0:=\{\vecb{v}_h\in\vecb{V}_h:\nabla_h\cdot\vecb{v}_h=0\},
      $$
      $$
         \vecb{V}_h^{\perp}:=\{\vecb{u}_h\in\vecb{V}_h:(\nabla_h\vecb{u}_h,\nabla_h\vecb{v}_h)=0\quad\forall\vecb{v}_h\in\vecb{V}_h^0\}.
      $$
      {Note that $\Pi_h^R\vecb{v}_h\in \vecb{H}(\mathrm{div};\Omega)$}. Testing the equation (\ref{SchemeCR}) with arbitrary $\vecb{v}_h\in\vecb{V}_h^0$ and $\vecb{f}=\nabla\phi$ yields
      \begin{align*}
         (\nabla\phi,\Pi_h^{R}\vecb{v}_h)=-(\phi,\nabla\cdot\Pi_h^{R}\vecb{v}_h)=-(\phi,\nabla_h\cdot\vecb{v}_h),
      \end{align*}
      \begin{align*}
         \mu(\nabla_h\vecb{u}_h, \nabla_h\vecb{v}_h) = -(\mu+\lambda)(\nabla_h\cdot\vecb{u}_h, \nabla_h\cdot\vecb{v}_h)-(\phi,\nabla_h\cdot\vecb{v}_h)=0,
      \end{align*}
      thus $\vecb{u}_h\in\vecb{V}_h^{\perp}$. 
      Consider $\vecb{v}_h=\vecb{u}_h$, integration by parts for the right hand side gives
      \begin{equation} \label{eqGrad1}
         \mu(\nabla_h\vecb{u}_h, \nabla_h\vecb{u}_h) + (\mu+\lambda)(\nabla_h\cdot\vecb{u}_h, \nabla_h\cdot\vecb{u}_h) = -(\phi,\nabla\cdot\Pi_h^{R}\vecb{u}_h).
      \end{equation}
      Using the fact that $\nabla\cdot\Pi_h^{R}\vecb{u}_h=\nabla_h\cdot\vecb{u}_h$ (see \eqref{commute}), Eq. (\ref{eqGrad1}) implies
      \begin{equation} \label{eqGrad2}
         \mu\Vert\nabla_h\vecb{u}_h\Vert^2 + (\mu+\lambda)\Vert\nabla_h\cdot\vecb{u}_h\Vert^2 \leq \Vert\phi\Vert \Vert \nabla_h\cdot\vecb{u}_h\Vert.
      \end{equation}
      {Similarly to \cite[Lemma~3.58]{john2016finite}}, it holds
      \begin{align*}
         \Vert\nabla_h\vecb{u}_h\Vert \leq {\frac{1}{\beta_{cr}}\Vert\nabla_h\cdot\vecb{u}_h\Vert},
      \end{align*}
      with $\beta_{cr}$ being the discrete inf-sup constant for the CR element. Thus, {applying the fact that $\|\nabla_h\cdot\vecb{u}_h\|\leq \sqrt{d}\|\nabla_h\vecb{u}_h\|$}, formula (\ref{eqGrad2}) can be rewritten as
      \begin{align*}
         \mu\Vert\nabla_h\vecb{u}_h\Vert^2 + (\mu+\lambda)\beta_{cr}^2\Vert\nabla_h\vecb{u}_h\Vert^2 \leq \sqrt{d} \Vert\phi\Vert\Vert\nabla_h\vecb{u}_h\Vert.
      \end{align*}
      Dividing by $\Vert\nabla_h\vecb{u}_h\Vert$, one gets
      \begin{align*}
         \Vert\nabla_h\vecb{u}_h\Vert \leq {\frac{\sqrt{d}}{(1+\beta_{cr}^2)\mu+\lambda\beta_{cr}^2}}\Vert\phi\Vert,
      \end{align*}
      which demonstrates that $\Vert\nabla_h\vecb{u}_h\Vert=\mathcal{O}(\lambda^{-1})$.
   \end{proof}
   \begin{figure}[htbp]
      \centering
      \includegraphics[scale=0.5]{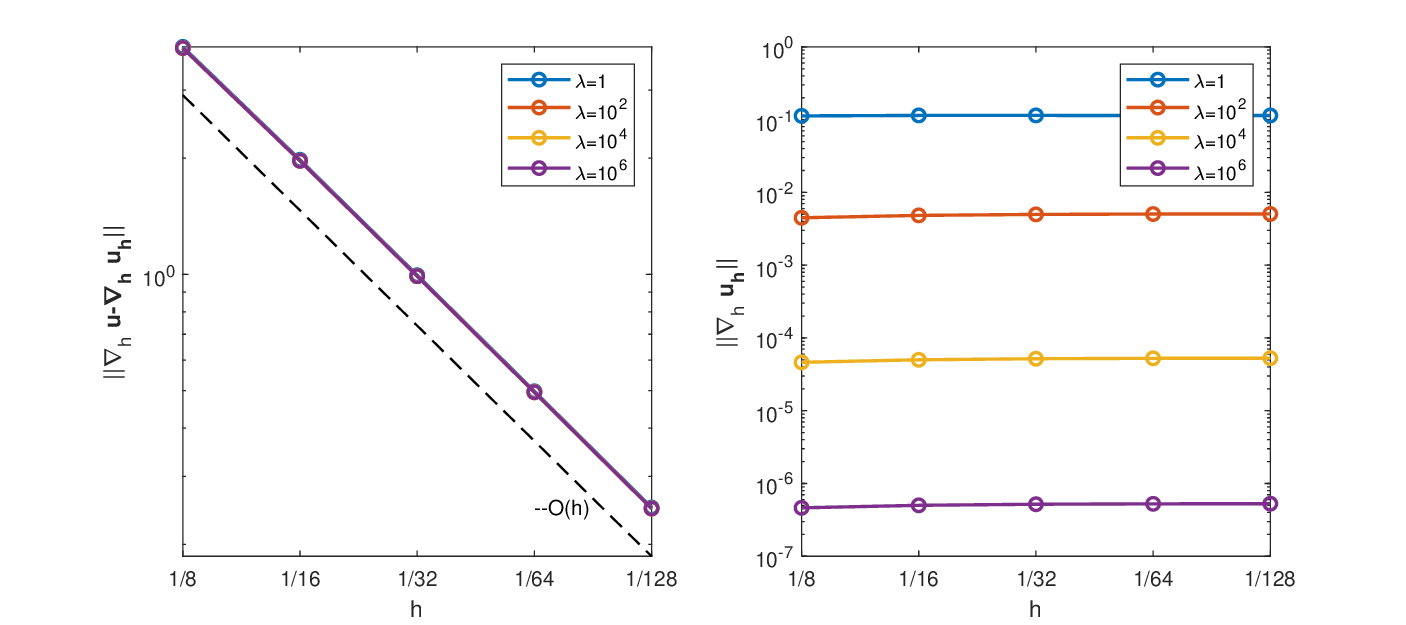}
      \caption{CR element for Example \ref{example1} (left) and Example \ref{example2} (right).}
      \label{figCR}
   \end{figure}

   \section{The \texorpdfstring{$\vecb{P}_1\oplus\vecb{RT}_0$}{} finite element schemes} \label{Section3}
   In this section we propose an $\vecb{H}(\mathrm{div})$-conforming finite element method for {the} linear elasticity problem with homogeneous displacement boundary condition.
   The $\vecb{P}_1\oplus\vecb{RT}_0$ finite element is proposed by Li and Rui \cite{li2022low} for Stokes flow,
   and it is the continuous vector-valued piecewise linear polynomial space $(\vecb{P}_1)$ enriched by the lowest-order Raviart-Thomas space $(\vecb{RT}_0)$.
   Let $\{\mathcal{T}_h\}$ be a family of triangluations of $\Omega$. Let $h_T$ and $h_e$ denote the diameters of elements $T$ and faces $e$, respectively, and $h:=\max_{T\in\mathcal{T}_h}h_T$.
   The set of interior faces and boundary faces of $\mathcal{T}_h$ are denoted
   by $\mathscr{E}^0$ and $\mathscr{E}^{\partial}$, respectively, and $\mathscr{E}:=\mathscr{E}^0 \cup \mathscr{E}^{\partial}$.
   An unit normal vector to the face $e$ is denoted by $\vecb{n}_e$. The family of meshes $\{\mathcal{T}_h\}$ is assumed to be shape-regular, that is, 
   there exists a constant $\gamma$, independent of $h$, such that
   \begin{align}\label{meshregularity}
      \frac{h_T}{\rho_T}\leq \gamma \quad \forall T\in\mathcal{T}_h,
   \end{align}
   where $\rho_T$ denotes the diameter of the largest ball contained in $T$.
   We define the space
   $\vecb{H}_0(\mathrm{div};\Omega):=\{\vecb{v}\in \vecb{H}(\mathrm{div};\Omega):\vecb{v}\cdot\vecb{n}=0 \text{ on } \partial\Omega\}$.
   For easier understanding of notations, the piecewise linear polynomial space is renamed as
   $$
      \vecb{V}_{h}^{1}:=\left\{\vecb{v} \in \vecb{V}:\left.\vecb{v}\right|_{T} \in\left[P_{1}(T)\right]^{d} \quad \forall T \in \mathcal{T}_{h}\right\}.
   $$
   %\begin{displaymath}
   %{H}_{0}(\operatorname{div};\Omega)=\{\vecb{v}\in {H}(\operatorname{div};\Omega): \vecb{v}\cdot \vecb{n}=0\text{ on } \Gamma\}.
   %\end{displaymath}
   The lowest-order Raviart--Thomas {finite element space \cite{boffi2013mixed} is denoted by}
   \begin{align}\label{def:pihr}
      \vecb{V}_h^R:=\left\{\vecb{v} \in \vecb{H}_0(\mathrm{div};\Omega):\left.\vecb{v}\right|_{T} \in\left[P_{0}(T)\right]^{d} \oplus \vecb{x} P_{0}(T) \quad\forall T\in\mathcal{T}_{h} \right\}.
   \end{align}
   The {space of piecewise constants reads}
   $$
      W_{h}:=\left\{q \in W:\left.q\right|_{T} \in P_{0}(T) \quad \forall T \in \mathcal{T}_{h}\right\}.
   $$
   The nodal interpolation is denoted by $\Pi_h^1:\vecb{V}\cap \vecb{C}^0(\Omega)\rightarrow \vecb{V}^1_h$. Moreover, we define
   the Raviart-Thomas interpolation $\Pi_h^R:\vecb{V}\rightarrow \vecb{V}_h^R$ and the
   orthogonal $L^2$ projection $P_h:W\rightarrow W_h$ by
   \begin{equation}\label{eq:pihr}
      ((\vecb{v}-\Pi_h^R\vecb{v})\cdot\vecb{n}_e,1)_e=0 \quad \forall e\in\mathscr{E},
   \end{equation}
   and
   \begin{equation} \label{propePh1}
      (r-P_hr,w_h)=0 \quad\forall w_h\in W_h,
   \end{equation}
   respectively. The following approximation and commutative properties {can be found in \cite{brenner2008mathematical,boffi2013mixed}}:
%   \begin{equation} \label{propePi11}
%      \Vert\vecb{v}-\Pi_h^1\vecb{v}\Vert_T+h_T\vert\vecb{v}-\Pi_h^1\vecb{v}\vert_{1,T}\leq Ch_T\vert\vecb{v}\vert_{2,T} \quad\forall T\in\mathcal{T}_h,
%   \end{equation}
%   \begin{align*}
%      \Vert\vecb{v}-\Pi_h^R\vecb{v}\Vert_{0,T} \leq Ch\vert\vecb{v}\vert_{1,T} \quad \forall T\in\mathcal{T}_h,
%   \end{align*}
%   \begin{equation} \label{propePh2}
%      \Vert r-P_hr\Vert_{0,T}\leq C{h^s}\Vert r\Vert_{s,T} \quad\forall T\in\mathcal{T}_h, \quad s=0,1,
%   \end{equation}
%   \begin{equation}\label{commute}
%      \nabla\cdot\Pi_h^R\vecb{v}=P_h\nabla\cdot\vecb{v}.
%   \end{equation}
   \begin{align}
   	\label{propePi11}
   	\Vert\vecb{v}-\Pi_h^1\vecb{v}\Vert_T+h_T\vert\vecb{v}-\Pi_h^1\vecb{v}\vert_{1,T}&\leq Ch_T^2\vert\vecb{v}\vert_{2,T} \quad&&\forall T\in\mathcal{T}_h,\\
   	\nonumber \Vert\vecb{v}-\Pi_h^R\vecb{v}\Vert_{0,T} &\leq Ch_T\vert\vecb{v}\vert_{1,T} \quad &&\forall T\in\mathcal{T}_h,\\
   	\label{propePh2}
   	\Vert r-P_hr\Vert_{0,T}&\leq C{h_T^s}\vert r\vert_{s,T} \quad&&\forall T\in\mathcal{T}_h, \quad s=0,1,\\
   	\label{commute}
   	\nabla\cdot\Pi_h^R\vecb{v}&=P_h\nabla\cdot\vecb{v}.
   \end{align}
   By \cite[Lemma~2.1]{li2022low}, we know $\vecb{V}_h^1\cap\vecb{V}_h^R=\{\vecb{0}\}$.
   Then the $\vecb{P}_1\oplus\vecb{RT}_0$ finite element space $\vecb{V}_h$ is a direct sum of these two spaces, i.e., $\vecb{V}_{h}:=\vecb{V}_{h}^{1}\oplus\vecb{V}_h^R$.
   For any $\vecb{u}_h\in\vecb{V}_h$, it can be uniquely decomposed into $\vecb{u}_h^1+\vecb{u}_h^R$, where $\vecb{u}_h^1\in\vecb{V}_{h}^{1}$ and $\vecb{u}_h^R\in\vecb{V}_h^R$.
   We know that $\vecb{V}_h$ is an $\vecb{H}(\operatorname{div})$-conforming space and $\vecb{V}_h\times W_h$ is a divergence-free pair in the sense of \cite{john2017divergence}, i.e., $\nabla\cdot\vecb{V}_h=W_h$.
   Moreover, we define
   $$
      \begin{aligned}
         \vecb{V}_h^0: & =\left\{\vecb{v}_h\in \vecb{V}_h:(\nabla\cdot\vecb{v}_h,q_h)=0\quad\forall q_h\in W_h\right\} \\
                       & =\left\{\vecb{v}_h\in \vecb{V}_h:\nabla\cdot\vecb{v}_h=0\right\}.
      \end{aligned}
   $$

   %$V_h=\vecb{V}_h^0 \oplus \vecb{V}_h^{\perp}$.
   \noindent\textbf{$\vecb{P}_1\oplus\vecb{RT}_0$ scheme 1.} Based on this element, we propose the following finite element scheme of (\ref{M1}):
   Find $\vecb{u}_h\in\vecb{V}_h$ such that
   \begin{equation}\label{mfM1} \tag{S1}
      {a_{h1}(\vecb{u}_{h},\vecb{v}_{h}):=}2\mu a_{h}(\vecb{u}_{h}, \vecb{v}_{h})+\lambda(\nabla\cdot\vecb{u}_{h},\nabla\cdot\vecb{v}_{h})=(\vecb{f}, \vecb{v}_{h}) \quad\forall\vecb{v}_{h}\in\vecb{V}_{h},
   \end{equation}
   where $a_{h}(\vecb{u}_{h}, \vecb{v}_{h})=a(\vecb{u}_{h}^{1}, \vecb{v}_{h}^{1})+a^{R}(\vecb{u}_{h}^{R}, \vecb{v}_{h}^{R})$.
   Let $\vecb{\psi}_e$ {denote} the Raviart-Thomas basis function {for the} face $e$ {such that} $\vecb{u}_h^R\in \vecb{V}_h^R$ can be rewritten as $\vecb{u}_h^R=\sum_{e\in\mathscr{E}^0}u_e\vecb{\psi}_e$. 
   Following \cite{li2022low}, $a^{R}(\cdot,\cdot)$ has three {choices} as follows
   \begin{align*}
      a^R(\vecb{u}_h^R, \vecb{v}_h^R) & =a^0(\vecb{u}_h^R, \vecb{v}_h^R):=\sum_{T\in\mathcal{T}_h}\alpha_T h_T^{-2}(\vecb{u}_h^R, \vecb{v}_h^R)_T,                                                                           \\
      a^R(\vecb{u}_h^R, \vecb{v}_h^R) & =a^D(\vecb{u}_h^R, \vecb{v}_h^R):=\sum_{T\in\mathcal{T}_h}\sum_{e\in \partial T\cap\mathscr{E}^0}\alpha_T h_T^{-2}u_ev_e(\vecb{\psi}_e,\vecb{\psi}_e)_T,                             \\
      a^R(\vecb{u}_h^R, \vecb{v}_h^R) & =a^{\mathrm{div}}(\vecb{u}_h^R, \vecb{v}_h^R):=\sum_{T\in \mathcal{T}_h}\sum_{e\in\partial T\cap\mathscr{E}^0}\alpha_T u_ev_e(\nabla\cdot\vecb{\psi}_e, \nabla\cdot\vecb{\psi}_e)_T.
   \end{align*}
   The three forms are spectrally equivalent (see \cite[Lemma 3.2]{li2022low}) and the parameters $\alpha_T$ are positive constants. Relevant proof can be found in \cite{li2022low}.
   For brevity, we choose $a^R=a^0$ in analysis.

   Analogously to the continuous setting, we can define the orthogonal complement of $\vecb{V}_h^0$ {with respect to} the bilinear form $a_h(\cdot,\cdot)$
   $$
      \vecb{V}_h^{\perp}:=\left\{\vecb{u}_h\in \vecb{V}_h: a_h(\vecb{u}_h,\vecb{v}_h)=0\quad\forall \vecb{v}_h\in \vecb{V}_h^0\right\}.
   $$
   % A block scheme of (\ref{mfM1}) can be written as: Find $(\vecb{u}_h^1,\vecb{u}_h^R)\in\vecb{V}_h^1\times \vecb{V}_h^R$ such that
   % \begin{equation} \label{blockform}
   % \begin{pmatrix}A_{LL} & 0\\0 & A_{RR}\end{pmatrix}\begin{pmatrix}U_L \\ U_R\end{pmatrix}\quad + \quad\begin{pmatrix}B_{LL} & B_{RL}\\B_{LR} & B_{RR}\end{pmatrix}\begin{pmatrix}U_L \\ U_R\end{pmatrix} \quad =\quad \begin{pmatrix}F_L \\ F_R \end{pmatrix},
   % \end{equation}
   % where $U_L$ and $U_R$ denote the unknown vectors corresponding to the continuous linear component and the RT component of the displacement $\vecb{u}_h$. The blocks $A_{LL}$, $A_{RR}$, $B_{LL}$, $B_{RL}$, $B_{LR}$, $B_{RR}$, $F_L$ and $F_R$ correspond to the bilinear forms $2\mu a(\vecb{u}_{h}^{1}, \vecb{v}_{h}^{1})$, $2\mu a^R(\vecb{u}_{h}^{R},\vecb{v}_{h}^{R})$, $\lambda(\nabla\cdot\vecb{u}_h^1,\nabla\cdot\vecb{v}_h^1)$, $\lambda(\nabla\cdot\vecb{u}_h^R,\nabla\cdot\vecb{v}_h^1)$,
   % $\lambda(\nabla\cdot\vecb{u}_h^1,\nabla\cdot\vecb{v}_h^R)$,
   % $\lambda(\nabla\cdot\vecb{u}_h^R,\nabla\cdot\vecb{v}_h^R)$,
   % $(\vecb{f},\vecb{v}_h^1)$ and $(\vecb{f},\vecb{v}_h^R)$, respectively.

   \begin{remark}
      The construction of {the} $\vecb{P}_1\oplus \vecb{RT}_0$ element is similar to {the} Bernardi--Raugel element. Both of them are based on continuous linear polynomials and supplemented with some stable functions to satisfy the inf-sup condition. The difference is that {the} $\vecb{P}_1\oplus \vecb{RT}_0$ element uses lowest-order Raviart--Thomas edge functions, while the BR element uses quadratic bubble functions. Thus {the} $\vecb{P}_1\oplus\vecb{RT}_0\times W_h$ {pair} is a divergence-free and pressure-robust pair.
      And compared to \textbf{BR {scheme}} (\ref{SchemeBR}), in our scheme (\ref{mfM1}) the term with $\lambda$ does not require $L^2$ projection.
      Moreover, the proposed scheme is easy to implement, it does not involve any face integrals.
   \end{remark}

   We define a larger space $\vecb{V}(h):=\vecb{V}\oplus\vecb{V}_h^R$ for analysis. For all $\vecb{v}\in\vecb{V}(h)$ we similarly have the unique decomposition $\vecb{v}=\vecb{v}^1+\vecb{v}^R$, where $\vecb{v}^1\in \vecb{V}$ and $\vecb{v}^R\in \vecb{V}_h^R$. We define the following norms or seminorms on $\vecb{V}(h)$:
   $$
      \Vert\vecb{v}\Vert^2_R := a^R(\vecb{v}^R, \vecb{v}^R),\quad
      \Vert\vecb{v}\Vert^2_h := a_h\left(\vecb{v}, \vecb{v}\right),\quad
      {\Vert\vecb{v}\Vert^2_{h1} := a_{h1}\left(\vecb{v}, \vecb{v}\right)}.
   $$
   Because of Korn's inequality (\ref{eqKorn}), $\Vert\cdot\Vert_h$ and $\Vert\cdot\Vert_{h1}$ are two norms. 
   Moreover, we define the interpolation $\Pi_h:\vecb{V}\cap\vecb{C}^0(\bar{\Omega})\rightarrow\vecb{V}_h$ as
   \begin{equation} \label{ConstructPi}
      \Pi_h\vecb{v}:=\Pi_h^1\vecb{v}+\Pi_h^R(\vecb{v}-\Pi_h^1\vecb{v}).
   \end{equation}
   {Let $\Vert\cdot\Vert_{h,T}$ and $\Vert\cdot\Vert_{h1,T}$ be the elementwise counterparts of $\Vert\cdot\Vert_h$ and $\Vert\cdot\Vert_{h1}$, respectively, 
   such that $\Vert\cdot\Vert_h^2=\sum_{T\in \mathcal{T}_h}\Vert\cdot\Vert_{h,T}^2$ and $\Vert\cdot\Vert_{h1}^2=\sum_{T\in \mathcal{T}_h}\Vert\cdot\Vert_{h1,T}^2$.}
   It was proven in \cite{li2022low} that
   \begin{align} 
   	\label{propePi1}
    \nabla \cdot \Pi_h \boldsymbol{v}&=P_h \nabla \cdot \boldsymbol{v} \quad&&{\forall \vecb{v}\in \vecb{V}\cap\vecb{C}^0(\bar{\Omega})},\\
    \label{propePi2}
    \left\|\boldsymbol{v}-\Pi_h \boldsymbol{v}\right\|_T + h_T\Vert\boldsymbol{v}-\Pi_h \boldsymbol{v}\Vert_{h,T}&\leq Ch_T^2| \boldsymbol{v}|_{2,T}\quad&&{\forall \vecb{v}\in [H^2(T)]^d, T\in\mathcal{T}_h}.
   \end{align}
   Thus we obtain
	\begin{equation} \label{propePi3}
	\begin{aligned}
	\Vert\boldsymbol{v}-\Pi_h \boldsymbol{v}\Vert_{h1,T} & \leq (2\mu)^{\frac{1}{2}}\Vert\boldsymbol{v}-\Pi_h \boldsymbol{v}\Vert_{h,T} + \lambda^{\frac{1}{2}}\Vert\nabla\cdot(\boldsymbol{v}-\Pi_h \boldsymbol{v})\Vert_{0,T}      \\
	& \leq(2\mu)^{\frac{1}{2}}\Vert\boldsymbol{v}-\Pi_h \boldsymbol{v}\Vert_{h,T} + \lambda^{\frac{1}{2}}\Vert\nabla\cdot\boldsymbol{v}-P_h\nabla\cdot\boldsymbol{v}\Vert_{0,T} \\
	&\leq Ch_T({(2\mu)^{\frac{1}{2}}}\vert\boldsymbol{v}\vert_{2,T}+\lambda^{\frac{1}{2}}\vert\nabla\cdot\boldsymbol{v}\vert_{1,T}) \quad\forall \vecb{v}\in [H^2(\Omega)]^d.
	\end{aligned}
	\end{equation}
   %The finite element scheme of (\ref{M2}) is presented as follows:
   %Find $\left(\vecb{u}_{h}, p_{h}\right) \in \vecb{V}_{h} \times W_{h}$ such that
   %\begin{equation} \label{mf1}
   %a_{h}\left(\vecb{u}_{h}, \vecb{v}_{h}\right) + b\left(\vecb{v}_{h}, p_{h}\right)=\left(\vecb{f}, \vecb{v}_{h}\right), \quad \forall \vecb{v}_{h} \in \vecb{V}_{h},
   %\end{equation}
   %\begin{equation} \label{mf2}
   %b\left(\vecb{u}_{h}, q_{h}\right)-\frac{1}{\lambda} c(p_h, q_h) =0, \quad \forall q_{h} \in W_{h},
   %\end{equation}
   %The mixed element pair satisfies $\nabla\cdot\vecb{V}_h=W_h$. For the incompressible elasticity problem, $\lambda\rightarrow\infty$, the displacement solution $\vecb{u}_h=(\vecb{u}_h^1,\vecb{u}_h^R)$ is divergence-free, i.e., $\nabla\cdot(\vecb{u}_h^1+\vecb{u}_h^R)=0$. From \cite{li2022low}, the additional enriched degrees of freedom can be eliminated by static condensation. The system can be regarded as a $\vecb{P}_1-P_0$ discretization with the stabilization term.
   \begin{lemma}[Inf-Sup Stability] \label{lemInfsup}
      There exists a positive constant $\beta_{is}$, {dependent on $\beta$, $\gamma$ and $\alpha_T$, $T\in\mathcal{T}_h$, but independent of $h$}, satisfying the inf-sup condition
      \begin{equation} \label{infsup}
         \sup_{\vecb{v}_h \in \vecb{V}_h\setminus\{\vecb{0}\}} \frac{(\nabla \cdot \vecb{v}_h, q_h)}{\Vert\vecb{v}\Vert_h} \geq \beta_{is}\left\|q_h\right\| \quad\forall q_h\in W_h,
      \end{equation}
      {where $\beta$ and $\gamma$ are the constants in \eqref{continuousinfsup} and \eqref{meshregularity}, respectively.}
      In addition, for all $q_h\in W_h$ there exists a unique $\vecb{u}_h^{\perp}\in\vecb{V}_h^{\perp}$ such that
      \begin{equation} \label{infsup2}
         \nabla\cdot\vecb{u}_h^{\perp}=q_h, \quad \Vert\vecb{u}_h^{\perp}\Vert_{h}\leq \beta_{is}^{-1}\Vert q_h\Vert.
      \end{equation}
   \end{lemma}
   \begin{proof}
      A similar inf-sup condition has already been proven in \cite{li2022low}:
      \begin{align*}
         \sup_{\vecb{v}_h\in\vecb{V}_h\setminus\{\vecb{0}\}}\frac{(\nabla\cdot\vecb{v}_h, q_h)}{|||\vecb{v}_h|||}\geq\beta_{is}\Vert q_h\Vert \quad\forall q_h\in W_h,
      \end{align*}
      where $|||\vecb{v}|||^2:=\Vert\nabla\vecb{v}^1\Vert^2+a^R(\vecb{v}^R,\vecb{v}^R)$ is a norm defined on $\vecb{V}(h)$. Then (\ref{infsup}) follows from the fact that $\|\vecb{\epsilon}(\vecb{v}^1)\|\leq \|\nabla\vecb{v}^1\|$ for any $\vecb{v}^1\in \vecb{V}$.
      Further, since $\nabla\cdot\vecb{V}_h=W_h$, the statement concerning \eqref{infsup2} is a direct consequence of \cite[Lemma~3.12]{john2016finite}. This completes the proof.

      % Due to $\nabla\cdot\vecb{V}_h=W_h$, 
      % there exists an isomorphism between $\vecb{V}_h^{\perp}$ and $W_h$.
      % As a direct consequence of the inf-sup condition (\ref{infsup}), we know $\Vert\vecb{u}_h^{\perp}\Vert_{h}\leq \beta_{is}^{-1}\Vert q_h\Vert$ from \cite[Lemma~3.12]{john2016finite}.
      
      %The second inequality in (\ref{infsup2}) follows from Korn's inequality.
   \end{proof}

   \begin{lemma} \label{ThUnique}
      The numerical scheme (\ref{mfM1}) has unique solution $\vecb{u}_h\in\vecb{V}_h$.
   \end{lemma}
   \begin{proof}
      It is trivial to prove
      \begin{align} 
         a_{h1}(\vecb{v}_h,\vecb{v}_h)&=\Vert\vecb{v}_h\Vert^2_{h1} \hspace{-2cm} &&{\quad\forall\vecb{v}_h\in\vecb{V}_h}, \label{eqAcoer}\\
         a_{h1}(\vecb{u}_h,\vecb{v}_h)&\leq\Vert\vecb{u}_h\Vert_{h1}\Vert\vecb{v}_h\Vert_{h1} \hspace{-2cm} &&{\quad\forall \vecb{u}_h,\vecb{v}_h\in\vecb{V}_h}.\label{eqAbound}
      \end{align}
      So the unique solvability of (\ref{mfM1}) is established.
   \end{proof}

   \begin{theorem} \label{ThHomoGradRob}
      The finite element scheme (\ref{mfM1}) is gradient-robust in the sense of Definition \ref{gradRobDef}. 
      If the {right-hand side equals} $\vecb{f}=\nabla\phi$ for some $\phi\in H^1(\Omega)$, 
      then the solution $\vecb{u}_h=\vecb{u}_h^{0}+\vecb{u}_h^{\perp}\in \vecb{V}_h^0\oplus \vecb{V}_h^\perp$ of (\ref{mfM1}) satisfies
      \begin{equation} \label{gradRobTh}
         \vecb{u}_h^0=\vecb{0},\quad\Vert\vecb{u}_h^{\perp}\Vert_h \leq {\frac{\sqrt{2}\max\{1,C_{inv}/\sqrt{\alpha}\}}{2\mu+\lambda \beta_{is}^2}}\Vert\phi\Vert,
      \end{equation}
      where $\alpha:=\min_{T\in\mathcal{T}_h} \alpha_T$ and $C_{inv}$ is a constant such that the inverse estimate $\|\nabla\cdot\vecb{v}^R\|_T\leq C_{inv} h_T^{-1} \|\vecb{v}^R\|_T$ holds for any
      $\vecb{v}^R\in \vecb{V}_h^R$ and $T\in \mathcal{T}_h$.
   \end{theorem}
   \begin{proof}
      Considering $\vecb{v}_h=\vecb{u}_h^0$ in (\ref{mfM1}), by the $a_h(\cdot,\cdot)$-orthogonality, and integrating by parts for the right hand side we get
      \begin{align*}
         2\mu a_h(\vecb{u}_h^0,\vecb{u}_h^0)=(\nabla\phi,\vecb{u}_h^0)=-(\phi,\nabla\cdot\vecb{u}_h^0)=0,
      \end{align*}
     which implies $\vecb{u}_h^0=\vecb{0}$.
      Considering $\vecb{v}_h=\vecb{u}_h^{\perp}$ in (\ref{mfM1}), {it follows}
      \begin{align*}
         2\mu a_h(\vecb{u}_h^{\perp},\vecb{u}_h^{\perp})+\lambda(\nabla\cdot\vecb{u}_h^{\perp},\nabla\cdot\vecb{u}_h^{\perp})=(\nabla\phi,\vecb{u}_h^{\perp})=
         -(\phi,\nabla\cdot\vecb{u}_h^{\perp})\leq \Vert\phi\Vert\Vert\nabla\cdot\vecb{u}_h^{\perp}\Vert.
      \end{align*}
      By an inverse estimate and a triangle inequality, it holds for any $\vecb{v}\in \vecb{V}(h)$ that
      \begin{align*}
         \|\nabla\cdot\vecb{v}\|&\leq \|\nabla\cdot\vecb{v}^1\| + \|\nabla\cdot\vecb{v}^R\|
         \leq \|\vecb{\epsilon}(\vecb{v}^1)\| + (C_{inv}/\sqrt{\alpha})\|\vecb{v}\|_R \\
         &\leq \max\{1,C_{inv}/\sqrt{\alpha}\}(\|\vecb{\epsilon}(\vecb{v}^1)\|+\|\vecb{v}\|_R)\leq \sqrt{2} \max\{1,C_{inv}/\sqrt{\alpha}\} \|\vecb{v}\|_h.
      \end{align*}
      Lemma \ref{lemInfsup} implies that
      \begin{align*}
         \Vert\vecb{u}_h^{\perp}\Vert_h\leq\beta_{is}^{-1}\Vert\nabla\cdot\vecb{u}_h^{\perp}\Vert,
      \end{align*}
      which, together with the above two estimates, yields
      \begin{align*}
         (2\mu + \lambda \beta_{is}^2) \|\vecb{u}_h^\perp\|_h^2 \leq \sqrt{2} \max\{1,C_{inv}/\sqrt{\alpha}\} \|\phi\|\|\vecb{u}_h^\perp\|_h.
      \end{align*}
      Then \cref{gradRobTh} follows. This completes the proof.
   \end{proof}
   \begin{remark}
      The {scheme} (\ref{mfM1}) satisfies both locking-free and gradient-robust properties.
      The reconstruction operator in \textbf{CR} {scheme} (\ref{SchemeCR}) is used to map discretely divergence-free functions to divergence-free functions, while the {scheme} (\ref{mfM1}) does not require reconstructing.
   \end{remark}

   Next we analyze the consistency error caused by this $\vecb{H}(\mathrm{div})$-conforming $\vecb{P}_1\oplus \vecb{RT}_0$ element.
   Let $\vecb{u}$ be the true solution of \cref{M1}. The consistency error is {denoted} by
   $$\delta_{h1}(\vecb{u},\vecb{v}):=(\vecb{f},\vecb{v})-a_{h1}(\vecb{u},\vecb{v}).$$
   Note that $a_{h1}(\vecb{u},\vecb{v})=2\mu a(\vecb{u},\vecb{v}^1)
      +\lambda(\nabla\cdot\vecb{u},\nabla\cdot\vecb{v})$ and
   $(\vecb{f},\vecb{v})=\left( -\nabla\cdot(2\mu\vecb{\epsilon}(\vecb{u})
         +\lambda(\nabla\cdot\vecb{u})\mathbf{I}),\vecb{v} \right)$ for all
   $\vecb{v}\in \vecb{V}(h)$.
   {Integrating by parts for the $2\mu a(\vecb{u},\vecb{v}^1)$ term,} one obtains the consistency error:
   \begin{equation} \label{consiErr}
      \begin{aligned}
         \left|\delta_{h1}(\vecb{u},\vecb{v})\right| & = \left| (-2\mu\nabla\cdot\vecb{\epsilon}(\vecb{u}),\vecb{v}^R)\right| \leq 2\mu\sum_{T\in\mathcal{T}_h} h_T|\vecb{u}|_{2,T} h_T^{-1}\Vert\vecb{v}^R\Vert_{0,T} \\
         &\leq 2\mu\left(\sum_{T\in\mathcal{T}_h}h_T^2\vert\vecb{u}\vert_{2,T}^2\right)^{1/2}
         \left(\sum_{T\in\mathcal{T}_h}h_T^{-2}\Vert\vecb{v}^R\Vert_{0,T}^2\right)^{1/2} \\
         &\leq {(2\mu/\sqrt{\alpha})}h|\vecb{u}|_2\Vert\vecb{v}\Vert_{R}\leq{(2\mu/\sqrt{\alpha})}h|\vecb{u}|_2\Vert\vecb{v}\Vert_{h}.
      \end{aligned}
   \end{equation}
   {Here $\alpha$ is the same as in Theorem~\ref{ThHomoGradRob}}.

   \section{Error estimates} \label{Section4}
   \begin{theorem} \label{ThHomoErrorOrder}
      %{Let $\Omega$ be a convex polygon/polyhedron equipped with a shape-regular mesh $\mathcal{T}_h$}. 
      Let $\vecb{u}$ be the solution of (\ref{M1weak}) and $\vecb{u}_h$ be the solution of (\ref{mfM1}). Then assuming $\vecb{u}\in [H^2(\Omega)]^d$ and $\vecb{f}\in[L^2(\Omega)]^d$, we have the following error estimates:
      \begin{align}
         \Vert\vecb{u}-\vecb{u}_h\Vert_{h} &\leq C h\vert\vecb{u}\vert_2,\label{ieq:estimate1}\\
         \Vert\vecb{u}-\vecb{u}_h\Vert_{h1} &\leq Ch({(2\mu)^{\frac{1}{2}}}\vert\vecb{u}\vert_2+\lambda^{\frac{1}{2}}\vert\nabla\cdot\vecb{u}\vert_{1}).  \label{ieq:estimate2}
      \end{align}
      {If additionally $\Omega$ is convex and $d=2$}, one further has 
      \begin{equation*}
         \Vert\vecb{u}-\vecb{u}_h\Vert_{h1} \leq C h{((2\mu)^\frac{1}{2}+\lambda^{-\frac{1}{2}})}\Vert\vecb{f}\Vert.
      \end{equation*}
      In all estimates the constants $C$ are independent of $\lambda$, {$\mu$} and $h$.
   \end{theorem}
   \begin{proof}
      First, {we split the error into}
      $$
         \vecb{u}-\vecb{u}_h=\vecb{u}-\Pi_h\vecb{u}-(\vecb{u}_h-\Pi_h\vecb{u}):=\eta_{u}-\xi_{u}.
      $$
      Subtracting (\ref{mfM1}) from (\ref{M1weak}), we get the following error equation
      \begin{equation} \label{errorEquation}
         2\mu a_h(\vecb{u}-\vecb{u}_h,\vecb{v}_h) + \lambda(\nabla\cdot(\vecb{u}-\vecb{u}_h),\nabla\cdot\vecb{v}_h)=-\delta_{h1}(\vecb{u},\vecb{v}_h) {\quad\forall \vecb{v}_h\in\vecb{V}_h}.
      \end{equation}
      Setting $\vecb{v}_h=\xi_u$, we get
      \begin{equation} \label{errorEquation2}
         2\mu a_{h}(\xi_u,\xi_u)+\lambda(\nabla\cdot\xi_u,\nabla\cdot\xi_u) = 2\mu a_{h}(\eta_u,\xi_u)+\lambda(\nabla\cdot\eta_u,\nabla\cdot\xi_u)-\delta_{h1}(\vecb{u},\xi_u).
      \end{equation}
      A combination of {\cref{propePh1}}, \cref{propePi1} and the fact that $\nabla\cdot\xi_u\in W_h$ implies that $(\nabla\cdot\eta_u,\nabla\cdot\xi_u)=0$.
      Then it follows from the coercivity and the boundedness of $a_{h}(\cdot,\cdot)$, and the consistency error (\ref{consiErr}) that
      {
      \begin{align*}
      	2\mu\Vert\xi_u\Vert_{h}^2 \leq 2\mu\Vert\eta_u\Vert_{h}\Vert\xi_u\Vert_{h} + (2\mu/\sqrt{\alpha}) h\vert\vecb{u}\vert_2\Vert\xi_u\Vert_{h},
      \end{align*}
      which implies
      \begin{align*}
      	\Vert\xi_u\Vert_{h}\leq\Vert\eta_u\Vert_{h} + (\sqrt{\alpha})^{-1}h\vert\vecb{u}\vert_2.
      \end{align*}
      }
     {The estimate} \cref{ieq:estimate1} follows from a combination of the above inequality, the triangle inequality and the interpolation error with respect to $\Vert\cdot\Vert_{h}$.
     
     Let us prove \cref{ieq:estimate2}. 
     %{Compared to the $H^2$-regularity in other references, we add the parameter $\mu$ in the above equation. It is convenient to show that our scheme is robust with respect to $\mu$ in the following proof.}
     Note that $(2\mu)^{\frac{1}{2}}\Vert\cdot\Vert_h\leq\Vert\cdot\Vert_{h1}$. By Equation (\ref{errorEquation2}), the coercivity of $a_{h1}(\cdot,\cdot)$, the boundedness of $a_{h}(\cdot,\cdot)$, and the consistency error (\ref{consiErr}), we obtain
     \begin{align*}
        \Vert\xi_u\Vert_{h1}^2 \leq 2\mu\Vert\eta_u\Vert_{h}\Vert\xi_u\Vert_{h} + (2\mu/\sqrt{\alpha}) h\vert\vecb{u}\vert_2\Vert\xi_u\Vert_{h}
        \leq\Vert\eta_u\Vert_{h1}\Vert\xi_u\Vert_{h1} + (2\mu/\alpha)^{\frac{1}{2}}h\vert\vecb{u}\vert_2\Vert\xi_u\Vert_{h1},
     \end{align*}
     which implies
     \begin{align*}
        {\Vert\xi_u\Vert_{h1}\leq \Vert\eta_u\Vert_{h1} + (2\mu/\alpha)^{\frac{1}{2}}h\vert\vecb{u}\vert_2.}
     \end{align*}
     {The above estimate} gives, together with the approximation properties (\ref{propePi3}) and the triangle inequality,
     \begin{equation} \label{errorEstimateH1}
       \begin{aligned}
         \Vert\vecb{u}-\vecb{u}_h\Vert_{h1} & {\leq\Vert\vecb{u}-\Pi_h\vecb{u}\Vert_{h1}+\Vert\Pi_h\vecb{u}-\vecb{u}_h\Vert_{h1}} \\
         & \leq 2\Vert\vecb{u}-\Pi_h\vecb{u}\Vert_{h1} + {(2\mu/\alpha)^{\frac{1}{2}}}h\vert\vecb{u}\vert_2 \\
         & \leq Ch({(2\mu)^{\frac{1}{2}}}\vert\vecb{u}\vert_2+\lambda^{\frac{1}{2}}\vert\nabla\cdot\vecb{u}\vert_{1}).
         %& \leq C h{((2\mu)^\frac{1}{2}+\lambda^{-\frac{1}{2}})}\Vert\vecb{f}\Vert.
         \end{aligned}
      \end{equation}
      Then the last inequality in Theorem~\ref{ThHomoErrorOrder} follows immediately from the $H^2$-regularity estimate \cite{brenner1992linear},
     \begin{equation} \label{M1regular}
     	\Vert\vecb{u}\Vert_2 + \lambda\Vert\nabla\cdot\vecb{u}\Vert_1\leq C\Vert\vecb{f}\Vert,
     \end{equation}
     which holds true in case {$\Omega$ is a convex polygon in two dimensions}. Thus we complete the proof.
   \end{proof}

   Using duality argument it is not hard to {obtain the error estimate for $\vecb{L}^2$-norm} (cf. \cite{brenner1992linear,li2022low})
   \begin{align*}
      \Vert\vecb{u}-\vecb{u}_h\Vert \leq Ch^2\Vert\vecb{f}\Vert,
   \end{align*}
   when \eqref{M1regular} holds true. 
   Here we propose a strategy to get a sharper $\vecb{L}^2$ estimate for our method like \cref{ieq:estimate1},
   which is based on a specifically designed projection
   $\Pi_h^\beta: \vecb{V}\rightarrow \vecb{V}_h$, defined by
   \begin{align*}
      \Pi_h^\beta\vecb{v}:=\Pi_h^e \vecb{v}+\Pi_h^R(\vecb{v}-\Pi_h^e\vecb{v}),
   \end{align*}
   where $\Pi_h^e: \vecb{V}\rightarrow \vecb{V}_h^1$ is an elliptic projection defined by
   \begin{align*}
      a(\Pi_h^e\vecb{v},\vecb{w}_h):=a(\vecb{v},\vecb{w}_h) \quad \forall \vecb{w}_h\in \vecb{V}_h^1.
   \end{align*}
   From known theory of elliptic type projection one has
   \begin{align*}
      \|\vecb{v}-\Pi_h^e\vecb{v}\|+h \|\nabla(\vecb{v}-\Pi_h^e\vecb{v})\|
      \leq C\inf_{\vecb{w}_h\in  \vecb{V}_h^1} \|\nabla(\vecb{v}-\vecb{w}_h)\|.
   \end{align*}
   Analogously, following the same strategy one can prove that
   \begin{equation} \label{propePie1}
      \nabla \cdot \Pi_h^\beta \boldsymbol{v}=P_h \nabla \cdot \boldsymbol{v},
   \end{equation}
   \begin{equation} \label{propePie2}
      \|\boldsymbol{v}-\Pi_h^\beta \boldsymbol{v}\| + h\Vert\boldsymbol{v}-\Pi_h^\beta \boldsymbol{v}\Vert_{h}\leq Ch^2| \boldsymbol{v}|_{2},
   \end{equation}
   which is similar to \cref{propePi1} and \cref{propePi2}, but an elementwise estimate for $\Pi_h^\beta$ is not available.

   \begin{theorem} \label{ThHomoL2Order}
      Let $\vecb{u}$ be the solution of variational formulation (\ref{M1weak}) and $\vecb{u}_h$ be the solution of (\ref{mfM1}).
      {Under the assumption that \eqref{M1regular} holds}, one has
      \begin{equation}\label{ieq:estimate3}
         \Vert\vecb{u}-\vecb{u}_h\Vert \leq C{(1+2\mu)}h^2\vert\vecb{u}\vert_2,
      \end{equation}
      where $C$ is a positive constant independent of $\lambda$, {$\mu$} and $h$.
   \end{theorem}
   \begin{proof}
      First we introduce the following duality problem:
      $$
         \begin{aligned}
            -\nabla\cdot(2\mu \vecb{\epsilon}(\vecb{\phi})+\lambda(\nabla\cdot\vecb{\phi})\mathbf{I}) & =\Pi_h^\beta\vecb{u}-\vecb{u}_h &  & \text { in } \Omega,         \\
            \vecb{\phi}                                                                               & =\vecb{0}                       &  & \text { on } \partial\Omega.
         \end{aligned}
      $$
      Since $\Pi_h^\beta\vecb{u}-\vecb{u}_h\in [L^2(\Omega)]^d$, {from \eqref{M1regular}} we have the following regularity
      \begin{equation} \label{dualRegul}
         \Vert\vecb{\phi}\Vert_2 + \lambda\Vert\nabla\cdot\vecb{\phi}\Vert_1\leq C\Vert\Pi_h^\beta\vecb{u}-\vecb{u}_h\Vert.
      \end{equation}
      Multiplying $\vecb{v}\in\vecb{V}(h)$, and integrating by parts, we can get
      \begin{align*}
         a_{h1}(\vecb{v},\vecb{\phi})=(\Pi_h^\beta\vecb{u}-\vecb{u}_h,\vecb{v})-\delta_{h1}(\vecb{\phi},\vecb{v}) \quad\forall\vecb{v}\in\vecb{V}(h).
      \end{align*}
      {Taking $\vecb{v}=\Pi_h^\beta\vecb{u}-\vecb{u}_h$ gives}
      \begin{equation}\label{eq:erroreqforL2}
         \begin{aligned}
            \Vert\Pi_h^\beta\vecb{u}-\vecb{u}_h\Vert^2 & =a_{h1}(\Pi_h^\beta\vecb{u}-\vecb{u}_h,\vecb{\phi})+\delta_{h1}(\vecb{\phi},\Pi_h^\beta\vecb{u}-\vecb{u}_h)                                                                          \\
                                                       & =a_{h1}(\Pi_h^\beta\vecb{u}-\vecb{u}_h,\vecb{\phi}-\Pi_h\vecb{\phi})+a_{h1}(\Pi_h^\beta\vecb{u}-\vecb{u}_h,\Pi_h\vecb{\phi})+\delta_{h1}(\vecb{\phi},\Pi_h^\beta\vecb{u}-\vecb{u}_h) \\
         \end{aligned}
      \end{equation}
%      From (\ref{propePh1}) and (\ref{propePie1}), one has
%      \begin{align*}
%      	\Vert\nabla\cdot (\Pi_h^\beta\vecb{u}-\vecb{u}_h)\Vert^2 =\left( ,\nabla\cdot(\Pi_h^\beta\vecb{u}-\vecb{u}_h)\right)
%      \end{align*}
      
      Similarly, note that ${\nabla\cdot (\Pi_h^\beta\vecb{u}-\vecb{u}_h)\in W_h}$, from \cref{propePi1} one has
      \begin{align*}
         \vert a_{h1}(\Pi_h^\beta\vecb{u}-\vecb{u}_h,\vecb{\phi}-\Pi_h\vecb{\phi})\vert
          & =
         \vert {2\mu}a_{h}(\Pi_h^\beta\vecb{u}-\vecb{u}_h,\vecb{\phi}-\Pi_h\vecb{\phi}) \vert              \\
          & \leq {2\mu}\Vert\Pi_h^\beta\vecb{u}-\vecb{u}_h\Vert_h \Vert\vecb{\phi}-\Pi_h\vecb{\phi}\Vert_h
         \leq Ch{(2\mu)}\Vert\Pi_h^\beta\vecb{u}-\vecb{u}_h\Vert_h \vert\vecb{\phi}\vert_2.
      \end{align*}
      On the other hand, a combination of the definition of $a_{h1}$, $\Pi_h^e$ and $\delta_{h1}$,
      together with \cref{propePie1} and the fact $\nabla\cdot\Pi_h\vecb{\phi}\in W_h$,
      implies
      {\small
      \begin{align*}
         a_{h1}(\Pi_h^\beta\vecb{u}-\vecb{u}_h,\Pi_h\vecb{\phi})
          & =2\mu\left(a(\Pi_h^e\vecb{u}-\vecb{u}_h^1,(\Pi_h\vecb{\phi})^1)
         +a^R((\Pi_h^\beta\vecb{u})^R-\vecb{u}_h^R,(\Pi_h\vecb{\phi})^R)\right) \\
         &+\lambda(\nabla\cdot(\Pi_h^\beta\vecb{u}-\vecb{u}_h),
         \nabla\cdot\Pi_h\vecb{\phi})                                                     \\
          & =2\mu\left(a(\vecb{u}-\vecb{u}_h^1,(\Pi_h\vecb{\phi})^1)
         +a^R((\Pi_h^\beta\vecb{u})^R-\vecb{u}_h^R,(\Pi_h\vecb{\phi})^R)\right)   \\
         &+\lambda(\nabla\cdot(\vecb{u}-\vecb{u}_h),
         \nabla\cdot\Pi_h\vecb{\phi})                                                     \\
          & =a_{h1}(\vecb{u}-\vecb{u}_h,\Pi_h\vecb{\phi})+
         2\mu a^R\left((\Pi_h^\beta\vecb{u})^R,(\Pi_h\vecb{\phi})^R\right) \\
          & =-\delta_{h1}(\vecb{u},\Pi_h\vecb{\phi})+
         2\mu a^R\left((\Pi_h^\beta\vecb{u})^R,(\Pi_h\vecb{\phi})^R\right).
      \end{align*}
      }

      From the definition of $\Vert\cdot\Vert_R$ and $\Vert\cdot\Vert_{h}$ one has
      \begin{align}\label{ieq:relaRandh}
         \|\vecb{v}\|_R\leq \|\vecb{v}-\vecb{w}\|_h \quad\forall \vecb{w}\in \vecb{V}.
      \end{align}
      From \cref{consiErr} and \cref{ieq:relaRandh}, we have
      \begin{align*}
         \vert \delta_{h1}(\vecb{u},\Pi_h\vecb{\phi})\vert
          & \leq Ch{(2\mu)}\vert\vecb{u}\vert_2\|\Pi_h\vecb{\phi}\|_R \\
         & \leq Ch{(2\mu)}\vert\vecb{u}\vert_2\|\vecb{\phi}-\Pi_h\vecb{\phi}\|_h\leq Ch^2{(2\mu)}\vert\vecb{u}\vert_2\vert\vecb{\phi}\vert_2, \\
         \vert a^R\left((\Pi_h^\beta\vecb{u})^R,(\Pi_h\vecb{\phi})^R\right)\vert
          & \leq \|\Pi_h^\beta\vecb{u}\|_R\|\Pi_h\vecb{\phi}\|_R\leq C \|\vecb{u}-\Pi_h^\beta\vecb{u}\|_h\|\vecb{\phi}-\Pi_h\vecb{\phi}\|_h \\
         & \leq C h\|\vecb{u}-\Pi_h^\beta\vecb{u}\|_h\vert\vecb{\phi}\vert_2,                     \\
         \vert{\delta_{h1}}(\vecb{\phi},\Pi_h^\beta\vecb{u}-\vecb{u}_h)\vert
          & \leq C h{(2\mu)}\vert\vecb{\phi}\vert_2\|\Pi_h^\beta\vecb{u}-\vecb{u}_h\|_R \\
         & \leq C h{(2\mu)}(\|\vecb{u}-\Pi_h^\beta\vecb{u}\|_h+\|\vecb{u}-\vecb{u}_h\|_h)\vert\vecb{\phi}\vert_2.
      \end{align*}
      Substituting above estimates into \cref{eq:erroreqforL2} gives
      \begin{align*}
         \Vert\Pi_h^\beta\vecb{u}-\vecb{u}_h\Vert^2
         \leq C(h\Vert\vecb{u}-\vecb{u}_h\Vert_{h}+h\Vert\vecb{u}-\Pi_h^\beta\vecb{u}\Vert_{h} + h^2\vert\vecb{u}\vert_2){(2\mu)}\vert\vecb{\phi}\vert_2.
      \end{align*}
      By the regularity assumption (\ref{dualRegul}), the error estimates \cref{ieq:estimate1} and \cref{propePie2}, and the triangle inequality we get
      \begin{align*}
         \Vert\vecb{u}-\vecb{u}_h\Vert \leq C(1+2\mu)h^2\vert\vecb{u}\vert_2.
      \end{align*}
      Thus we complete the proof.
   \end{proof}

   \section{Mixed boundary conditions} \label{Section5}
   We define the space compatible with mixed boundary conditions (\ref{M2}).
   In case no ambiguity occurs, we use the same notations in Section \ref{Section4}. We define
   \begin{align}
   	\nonumber\vecb{H}_{\Gamma_D}(\mathrm{div};\Omega) &:=\left\{\vecb{v}\in \vecb{H}(\mathrm{div};\Omega):\vecb{v}\cdot\vecb{n}|_{\Gamma_D}=0 \right\},\\
   	\nonumber\vecb{V}_h^R &:=\left\{\vecb{v}_h\in \vecb{H}_{\Gamma_D}(\mathrm{div};\Omega):\vecb{v}_h|_{T}\in [P_0(T)]^d\oplus\vecb{x}P_0(T) \,\forall T\in\mathcal{T}_h\right\},\\
   	\nonumber\vecb{V}_h^1 &:=\left\{\vecb{v}_h\in\vecb{V}_{\Gamma_D}:\vecb{v}_h|_{T}\in[P_1(T)]^d\, \forall T\in\mathcal{T}_h\right\},\\
   	\nonumber\vecb{V}_h &:=\vecb{V}_h^1\oplus\vecb{V}_h^R, \vecb{V}(h):=\vecb{V}_{\Gamma_D}\oplus\vecb{V}_h^R.
   \end{align}
   \begin{remark}\label{rem:principle}
      The scheme for mixed boundary conditions should be designed carefully
      for $\vecb{P}_1\oplus \vecb{RT}_0$ element to obtain an optimally convergent consistency error.
      To be more precise, denote by $a_h^m(\cdot,\cdot)$ and $F(\cdot)$ two generic forms representing the
      left-hand side and right-hand side of a discretization, respectively.
      Let $\vecb{u}$ be the true solution related to \cref{M2}. The principle to designing $F$ and $a_h^m$ is that
      we hope they satisfy
      \begin{align}\label{eq:principle}
         F(\vecb{v})-a_h^m(\vecb{u},\vecb{v})=(-2\mu\nabla\cdot\vecb{\epsilon}(\vecb{u}),\vecb{v}^R)\text{ for all }\vecb{v}\in \vecb{V}(h),
      \end{align}
      like $\delta_{h1}$ in Section~\ref{Section3}. In this way the consistency error is still
      optimally convergent. A trivial extension from the pure Dirichlet problem to the mixed boundary problem might read
      \begin{align}\label{sch:badscheme}
         2\mu a_h(\vecb{u}_h,\vecb{v}_h)  + \lambda(\nabla\cdot\vecb{u}_h,\nabla\cdot\vecb{v}_h)
         =(\vecb{f},\vecb{v}_h) + \int_{\Gamma_N}\vecb{g}\cdot\vecb{v}_hds.
      \end{align}
      However, one can check that this does not satisfy the principle: some additional consistency error
      arises from the Neumann boundary part because the discretization related to $\vecb{RT}_0$ part in
      $a_h$ is not obtained from integration by parts. The numerical experiments later also demonstrate
      that the above scheme is not optimal. To overcome this issue, we should modify the discretization
      {in the case of mixed boundary conditions}. Several schemes which satisfy the principle {are} listed below.
   \end{remark}
   \noindent\textbf{$\vecb{P}_1\oplus\vecb{RT}_0$ scheme 2.}
   The nonsymmetric finite element scheme to deal with mixed boundary conditions (\ref{M2}) {reads}
   \begin{equation} \label{mfM2}\tag{S2}
      \begin{aligned}
         a_{NS}(\vecb{u}_h,\vecb{v}_h) &:=2\mu a_h(\vecb{u}_h,\vecb{v}_h) + \lambda(\nabla\cdot\vecb{u}_h,\nabla\cdot\vecb{v}_h)+\int_{\Gamma_N}2\mu\vecb{\epsilon}(\vecb{u}_h^1)\vecb{n}\cdot\vecb{v}^R_hds \\
         &-\int_{\Gamma_N}2\mu\vecb{\epsilon}(\vecb{v}_h^1)\vecb{n}\cdot\vecb{u}^R_hds=(\vecb{f},\vecb{v}_h) + \int_{\Gamma_N}\vecb{g}\cdot\vecb{v}_hds.
      \end{aligned}
   \end{equation}
   \noindent\textbf{$\vecb{P}_1\oplus\vecb{RT}_0$ scheme 3.}
   The symmetric finite element scheme to deal with mixed boundary conditions (\ref{M2}) {reads}
   \begin{equation} \label{mfM2_2}\tag{S3}
      \begin{aligned}
         a_{S}(\vecb{u}_h,\vecb{v}_h) &:=2\mu a_h(\vecb{u}_h,\vecb{v}_h) + \lambda(\nabla\cdot\vecb{u}_h,\nabla\cdot\vecb{v}_h)+\int_{\Gamma_N}2\mu\vecb{\epsilon}(\vecb{u}_h^1)\vecb{n}\cdot\vecb{v}^R_hds \\
         &+\int_{\Gamma_N}2\mu\vecb{\epsilon}(\vecb{v}_h^1)\vecb{n}\cdot\vecb{u}^R_hds =(\vecb{f},\vecb{v}_h) + \int_{\Gamma_N}\vecb{g}\cdot\vecb{v}_hds.
      \end{aligned}
   \end{equation}
   \noindent\textbf{$\vecb{P}_1\oplus\vecb{RT}_0$ scheme 4.}
   A modified version of \eqref{mfM2} or \eqref{mfM2_2} {reads}
   \begin{equation} \label{mfM2_3}\tag{S4}
      \begin{aligned}
        &a_{M}(\vecb{u}_h,\vecb{v}_h):=2\mu a_h(\vecb{u}_h,\vecb{v}_h) + \lambda(\nabla\cdot\vecb{u}_h,\nabla\cdot\vecb{v}_h)+\int_{\Gamma_N}(2\mu\vecb{\epsilon}(\vecb{u}_h^1)\vecb{n}\cdot\vecb{n})(\vecb{v}^R_h\cdot\vecb{n})ds \\
        &\pm\int_{\Gamma_N}(2\mu\vecb{\epsilon}(\vecb{v}_h^1)\vecb{n}\cdot\vecb{n})(\vecb{u}^R_h\cdot\vecb{n})ds=(\vecb{f},\vecb{v}_h) + \int_{\Gamma_N}\left[\vecb{g}\cdot\vecb{v}_h^1+(\vecb{g}\cdot\vecb{n})(\vecb{v}_h^R\cdot\vecb{n})\right]ds.
      \end{aligned}
   \end{equation}
   The main feature of \eqref{mfM2_3} is that it only involves the normal component of $\vecb{v}_h^R$ on the stress boundary,
   which matches the degrees of freedom of $\vecb{RT}_0$ well and, hence, makes the scheme easier to implement.
   It can be verified that all the three schemes satisfy the designing principle \cref{eq:principle} in Remark \ref{rem:principle}
   by integration by parts. For example, the consistency error of \eqref{mfM2} is
	\begin{equation}\label{eq:conserrNS}
	\begin{aligned}
	&\delta_{NS}(\vecb{u},\vecb{v}):=(\vecb{f},\vecb{v}) + \int_{\Gamma_N}\vecb{g}\cdot\vecb{v}ds-a_{NS}(\vecb{u}, \vecb{v}) \\
	& =\left(-\nabla\cdot(2\mu\vecb{\epsilon}(\vecb{u})+\lambda\nabla\cdot\vecb{u}\textbf{I}),\vecb{v}\right)
	+\int_{\Gamma_N}
	\left[2\mu\vecb{\epsilon}(\vecb{u})\vecb{n}\cdot\vecb{v}+\lambda\nabla\cdot\vecb{u}(\vecb{v}\cdot\vecb{n})\right]ds - a_{NS}(\vecb{u}, \vecb{v})\\
	& =2\mu a(\vecb{u},\vecb{v}^1)+\lambda(\nabla\cdot\vecb{u},\nabla\cdot\vecb{v})+
	\left(-\nabla\cdot(2\mu\vecb{\epsilon}(\vecb{u})),\vecb{v}^R\right)
	+\int_{\Gamma_N}
	2\mu\vecb{\epsilon}(\vecb{u})\vecb{n}\cdot\vecb{v}^Rds - a_{NS}(\vecb{u}, \vecb{v}) \\
	& =\left(-\nabla\cdot(2\mu\vecb{\epsilon}(\vecb{u})),\vecb{v}^R\right).
	\end{aligned}
	\end{equation}
   \begin{remark}
      Compared to \cref{sch:badscheme},
   the third term in the left-hand side of {schemes} \eqref{mfM2}--\eqref{mfM2_3} is introduced to satisfy the
   designing principle, while the fourth term is a consistent term to guarantee that a scheme is
   symmetric or nonsymmetric but stable as long as $\alpha_T, T\in \mathcal{T}_h$, are positive.
   This strategy is very similar to the discontinuous Galerkin (DG) methods for the elliptic problem \cite{arnold2002unified}.
   {However, there is a fundamental difference between our method and DG methods. 
   In contrast to the DG methods, the proposed schemes here do not involve any interior jump stabilization or
   face integral over interior faces, which are simpler to implement and do not
   change the sparsity pattern of the coefficient matrix.}
   \end{remark}

   The analysis of these schemes is indeed very similar to the pure Dirichlet boundary case. For brevity,
   we only analyze (\ref{mfM2}) and (\ref{mfM2_2}) below. We redefine norm $\|\cdot\|_h$ {on $\vecb{V}(h)$:}
%   \begin{align*}
%      \|\vecb{v}\|_h^2:=a(\vecb{v}^1,\vecb{v}^1) + \|\vecb{v}\|_R^2+ \sum_{e\in\Gamma_N}h_e\Vert\vecb{\epsilon}(\vecb{v}^1)\Vert^2_e.
%   \end{align*}
	\begin{equation} \label{eqNewh}
		\|\vecb{v}\|_h^2:=a(\vecb{v}^1,\vecb{v}^1) + a^R(\vecb{v}^R,\vecb{v}^R)+ \sum_{e\in\Gamma_N}h_e\Vert\vecb{\epsilon}(\vecb{v}^1)\Vert^2_e{,}
	\end{equation}
   where {$\vecb{v}^1\in\vecb{V}_{\Gamma_D}$, $\vecb{v}^R\in\vecb{V}_h^R$ and}  $\Vert\vecb{\epsilon}(\vecb{v}^1)\Vert_e^2:=\int_e\vecb{\epsilon}(\vecb{v}^1):\vecb{\epsilon}(\vecb{v}^1) ds$. Then we define a norm
   $$
      \Vert\vecb{v}\Vert^2_{h2}:=2\mu\Vert\vecb{v}\Vert^2_h + \lambda(\nabla\cdot\vecb{v},\nabla\cdot\vecb{v}).
   $$
   %where $\nabla_h\cdot$ denotes the piecewise divergence operator.
   {Assume that {$T\in\mathcal{T}_h$} is an element with $e$ as one edge}. By the trace inequality and the inverse inequality we can get
%   $$
%      h_e\Vert\vecb{\epsilon}(\vecb{v}^1)\Vert^2_e \leq C(\Vert\nabla\vecb{v}^1\Vert^2_T + h_T^2\Vert\nabla^2\vecb{v}^1\Vert^2_T).
%   $$
   $$
      {h_e\Vert\vecb{\epsilon}(\vecb{v}_h^1)\Vert^2_e \leq C\Vert\nabla\vecb{v}_h^1\Vert^2_T\quad\forall \vecb{v}_h\in\vecb{V}_h.}
   $$
   As a result, the two norms $\Vert\cdot\Vert_{h1}$ and $\Vert\cdot\Vert_{h2}$ are equivalent in $\vecb{V}_h$ space, i.e., $\Vert\vecb{v}_h\Vert_{h2}\leq C\Vert\vecb{v}_h\Vert_{h1}$.
   For any $\vecb{v}\in [H^2(T)]^d$ and $e\subset \partial T$, we have
   \begin{align*}
      h_e^{\frac{1}{2}}\Vert\vecb{\epsilon}(\vecb{v}-\Pi_h^1\vecb{v})\Vert_e
      \leq C(\vert\vecb{v}-\Pi_h^1\vecb{v}\vert_{1,T}
      +h\vert\vecb{v}-\Pi_h^1\vecb{v}\vert_{2,T})
      = C(\vert\vecb{v}-\Pi_h^1\vecb{v}\vert_{1,T}
      +h\vert\vecb{v}\vert_{2,T}),
   \end{align*}
   which means the interpolation estimate \cref{propePi2} still holds for new $\|\cdot\|_h$ norm {(\ref{eqNewh})},
   and together with (\ref{propePi3}), implies
   \begin{equation} \label{propePi4}
      \Vert\vecb{v}-\Pi_h\vecb{v}\Vert_{h2}\leq Ch({(2\mu)^{\frac{1}{2}}}\vert\vecb{v}\vert_2+\lambda^{\frac{1}{2}}\vert\nabla\cdot\vecb{v}\vert_1) \quad\forall \vecb{v}\in [H^2(\Omega)]^d.
   \end{equation}

   Define $\mathcal{T}_h(\Gamma_N):=\{T\in \mathcal{T}_h:\vert\partial T\cap \Gamma_N\vert\neq 0\}$. The following lemma is concerned with the unique solvability of (\ref{mfM2}) and (\ref{mfM2_2}).
   \begin{lemma}\label{lem:coerformixed}
      The numerical scheme (\ref{mfM2}) has a unique solution $\vecb{u}_h\in\vecb{V}_h$,
      and the numerical scheme (\ref{mfM2_2}) has a unique solution $\vecb{u}_h\in\vecb{V}_h$ if all $\alpha_T, T\in\mathcal{T}_h(\Gamma_N),$ are large enough.
   \end{lemma}
   \begin{proof}
      For any $\vecb{v}_h\in\vecb{V}_h$, we have
      $$
     \begin{aligned}
        a_{NS}(\vecb{v}_h,\vecb{v}_h) & =2\mu a_h(\vecb{v}_h,\vecb{v}_h) + \lambda(\nabla\cdot\vecb{v}_h,\nabla\cdot\vecb{v}_h)\\
        &\geq{\Vert\vecb{v}_h\Vert_{h1}^2}\geq C\Vert\vecb{v}_h\Vert_{h2}^2.
     \end{aligned}
      $$
      Thus the coercivity of $a_{NS}(\cdot,\cdot)$ holds. It follows from Cauchy-Schwarz inequality and trace inequality that
      \begin{equation} \label{boundEdgeInte}
         \begin{aligned}
            \left|\sum_{e\in\Gamma_N}\int_e2\mu\vecb{\epsilon}(\vecb{u^1})\vecb{n}\cdot \vecb{v}^R ds\right| & \leq 2\mu\left(\sum_{e\in\Gamma_N}h_e\Vert\vecb{\epsilon}(\vecb{u}^1)\Vert_e^2\right)^{\frac{1}{2}}\left(\sum_{e\in\Gamma_N}h_e^{-1}\Vert\vecb{v}^R\Vert_e^2\right)^{\frac{1}{2}} \\
            & \leq{2\mu}C\left(\sum_{e\in\Gamma_N}h_e\Vert\vecb{\epsilon}(\vecb{u}^1)\Vert_e^2\right)^{\frac{1}{2}} \left(\sum_{T\in\mathcal{T}_h(\Gamma_N)}h_T^{-2}\Vert\vecb{v}^R\Vert_T^2\right)^{\frac{1}{2}} \\
            & \leq{2\mu}C\Vert\vecb{u}\Vert_{h}\Vert\vecb{v}\Vert_{R} \quad\forall \vecb{u},\vecb{v}\in\vecb{V}(h).
         \end{aligned}
      \end{equation}
      The term $\int_{\Gamma_N}2\mu\vecb{\epsilon}(\vecb{v}_h^1)\vecb{n}\cdot\vecb{u}^R_hds$ can be bounded in the same way.
      Using the above inequality, we can get the boundedness of $a_{NS}(\cdot,\cdot)$. By Lax--Milgram Theorem, we get the unique solvability of (\ref{mfM2}).

      Next we analyze the unique solvability of (\ref{mfM2_2}). From the bound of {the} edge {integrals} (\ref{boundEdgeInte}) and Holder's inequality, we can get
      $$
         \begin{aligned}
            a_{S}(\vecb{v}_h,\vecb{v}_h) & \geq 2\mu(\vecb{\epsilon}(\vecb{v}_h^1),\vecb{\epsilon}(\vecb{v}_h^1)) + 2\mu\sum_{T\in\mathcal{T}_h}\alpha_{T}{h_{T}^{-2}}(\vecb{v}_h^{R}, \vecb{v}_h^{R})_T+\lambda(\nabla\cdot\vecb{v}_h,\nabla\cdot\vecb{v}_h) \\
            &-\left| 2\int_{\Gamma_N}2\mu\vecb{\epsilon}(\vecb{v}_h^1)\vecb{n}\cdot\vecb{v}^R_hds\right| \\
            &\geq 2\mu(\vecb{\epsilon}(\vecb{v}_h^1),\vecb{\epsilon}(\vecb{v}_h^1)) + 2\mu\sum_{T\in\mathcal{T}_h}\alpha_{T}{h_{T}^{-2}}(\vecb{v}_h^{R}, \vecb{v}_h^{R})_T + \lambda(\nabla\cdot\vecb{v}_h,\nabla\cdot\vecb{v}_h) \\
            &-\left(2\mu\varepsilon\sum_{T\in\mathcal{T}_h(\Gamma_N)}\Vert\nabla\vecb{v}_h^1\Vert_T^2 +2\mu C\sum_{T\in\mathcal{T}_h(\Gamma_N)}h_T^{-2}\Vert\vecb{v}_h^R\Vert_T^2\right) \\
            &= 2\mu\left((\vecb{\epsilon}(\vecb{v}_h^1),\vecb{\epsilon}(\vecb{v}_h^1))-\varepsilon\sum_{T\in\mathcal{T}_h(\Gamma_N)}\Vert\nabla\vecb{v}_h^1\Vert_T^2\right) \\
            &+2\mu\left(\sum_{T\in\mathcal{T}_h}\alpha_{T}{h_{T}^{-2}}(\vecb{v}_h^{R}, \vecb{v}_h^{R})_T-C\sum_{T\in\mathcal{T}_h(\Gamma_N)}h_T^{-2}\Vert\vecb{v}_h^R\Vert_T^2\right) +\lambda(\nabla\cdot\vecb{v}_h,\nabla\cdot\vecb{v}_h).
         \end{aligned}
      $$
      The constant $C$ depends on {the constants of} the trace inequality and Holder's inequality. {The parameters} $\alpha_T$ over $\mathcal{T}_h(\Gamma_N)$ should be chosen greater than $C$.
      The boundedness of $a_{S}(\cdot,\cdot)$ is very similar to $a_{NS}(\cdot,\cdot)$. Thus we complete the proof.
   \end{proof}
   % \begin{remark}
   %    In numerical schemes (\ref{mfM2}) and (\ref{mfM2_2}), the term $\int_{\Gamma_N}2\mu\vecb{\epsilon}(\vecb{u}_h^1)\vecb{n}\cdot\vecb{v}^R_hds$ is referred to consistency term, and $\pm\int_{\Gamma_N}2\mu\vecb{\epsilon}(\vecb{v}_h^1)\vecb{n}\cdot\vecb{u}^R_hds$ is symmetry or nonsymmetry term.
   %    The nonsymmetric scheme (\ref{mfM2}) is coercive as long as $\alpha_T>0$ for all $T\in\mathcal{T}_h$, while the coercivity of the symmetric scheme (\ref{mfM2}) is conditional, $\alpha_T$ in $\mathcal{T}_h(\Gamma_N)$ should be sufficiently large.
   %    Due to the symmetry of (\ref{mfM2_2}), this scheme can be solved by some fast algorithm, such as conjugate gradient method and blocks preconditioner methods.
   %    Unlike other nonconforming finite element method, our $\vecb{P1\oplus RT0}$ schemes preserve discrete Korn's inequality without stabilized term.
   % \end{remark}
   
   %分类讨论，d=2时参考文献，d=3时假设===========================================
   %混合边界时只有 dual-mixed {methods}\cite{gatica2006analysis,gatica2007dual,gatica2009augmented}
   %In \cite{gatica2006analysis,gatica2007dual,gatica2009augmented} they imposed Neumann boundary condition as a Lagrange multiplier in two dimensions. 
   Based on the coercivity and boundedness analyzed in Lemma~\ref{lem:coerformixed}, the consistency error such as
   \cref{eq:conserrNS}, the estimates of interpolation \cref{propePi4} and $H^2$-regularity in two dimensions, one can similarly obtain the following estimates.
   \begin{theorem} \label{ThMixErrorOrder}
      %{Let $\Omega$ be a convex polygon in $\mathbb{R}^2$ equipped with a shape-regular mesh $\mathcal{T}_h$}. 
      Let $\vecb{u}$ be the solution of (\ref{M2weak}) and $\vecb{u}_h$ be the solution of (\ref{mfM2}) or (\ref{mfM2_2}). Then assuming $\vecb{u}\in [H^2(\Omega)]^d$, we have the following error estimates:
      \begin{align*}
         \Vert\vecb{u}-\vecb{u}_h\Vert_{h} &\leq Ch\vert\vecb{u}\vert_2, \\
         {\Vert\vecb{u}-\vecb{u}_h\Vert_{h2}} &\leq {Ch((2\mu)^{\frac{1}{2}}\vert\vecb{u}\vert_2+\lambda^{\frac{1}{2}}\vert\nabla\cdot\vecb{u}\vert_1)},
      \end{align*}
      where $C$ is a positive constant independent of $\lambda$, {$\mu$} and $h$.
      % \begin{align*}
      %    \Vert\vecb{u}-\vecb{u}_h\Vert \leq Ch^2(\Vert\vecb{f}\Vert+\Vert\vecb{g}\Vert_{H^{1/2}(\Gamma_N)}).
      % \end{align*}
   \end{theorem}

   \section{Numerical {experiments}} \label{Section6}
   In this section, we divide into two subsections to verify the theoretical results in Theorem \ref{ThHomoGradRob}, Theorem \ref{ThHomoErrorOrder}, Theorem \ref{ThHomoL2Order} and
   Theorem \ref{ThMixErrorOrder}.
   The numerical experiments in the first subsection follows the examples in Section \ref{Section2}.
   The second subsection is the Cook's Membrane problem, which is used to show the robustness of our novel scheme for nearly incompressible elasticity.
   We set $\alpha_T=\alpha=1$ for all $T\in\mathcal{T}_h$.
   We choose $a^R(\cdot,\cdot)=a^{\mathrm{div}}(\cdot,\cdot)$, because it is related to a diagonal block, 
   and all the terms of the matrix can be calculated using the barycentric quadrature rule.

   \subsection{Parameter-robustness test}
   The examples to show the locking-free and gradient-robust properties are the same as {in} Section \ref{Section2}.
   We always use the primal formulation (\ref{M1}) to take numerical experiments.
   When we show the accuracy of our novel schemes, the boundary conditions are divided into two cases.
   One is homogeneous Dirichlet boundary condition $\vecb{u}=\vecb{0} \text{ on } \partial\Omega$, and the other is mixed boundary conditions (\ref{M2})
   $$
      \vecb{u}=\vecb{0} \,\text{ on }\Gamma_D,\quad\vecb{\sigma}\vecb{n}=\vecb{g} \,\text{ on }\Gamma_N.
   $$
   The Neumann boundary $\Gamma_N$ is posed on the right boundary ($x=1$) of the domain, while on the other three sides the Dirichlet condition is used.
   \begin{figure}[htbp]
   	\centering
   	\begin{subfigure}{.35\textwidth}
   		\centering
   		\includegraphics[width=\textwidth]{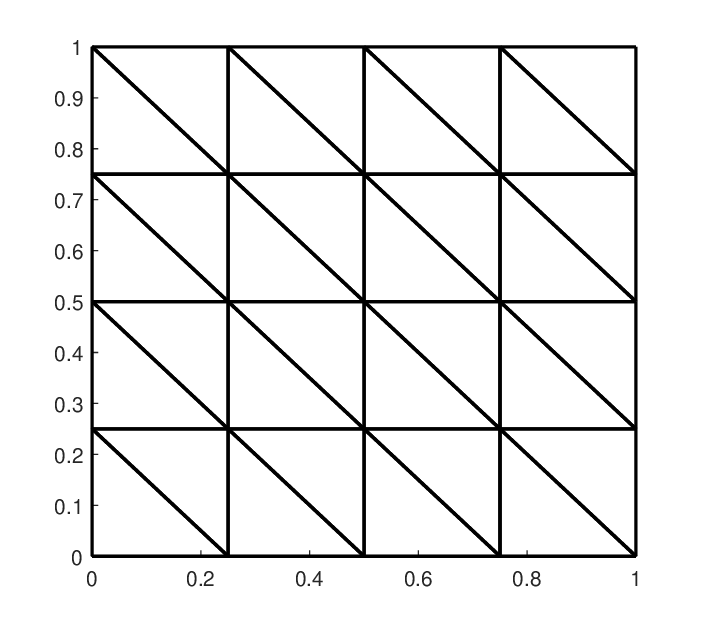}
   		\caption*{G1}
   	\end{subfigure}
   	\begin{subfigure}{.35\textwidth}
   		\centering
   		\includegraphics[width=\textwidth]{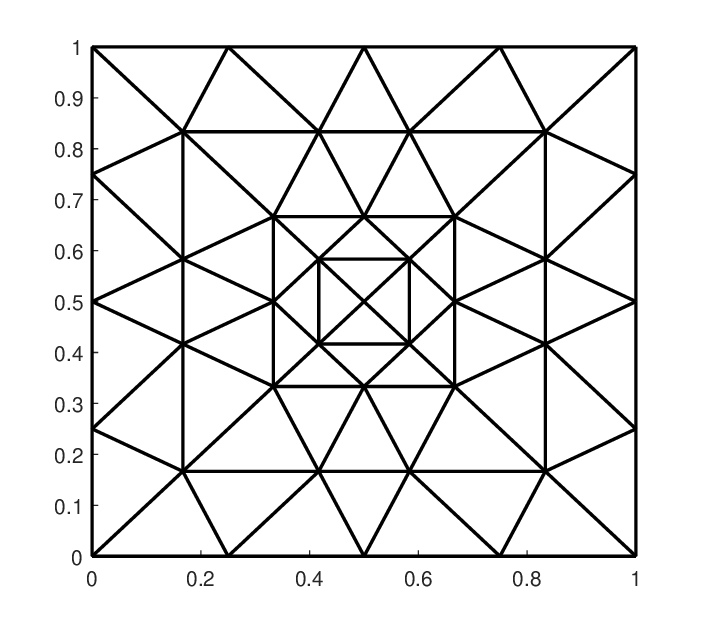}
   		\caption*{G2}
   	\end{subfigure}
   	\caption{Structured grid (G1) and unstructured grid (G2).} \label{meshGrid}
   \end{figure}

   \begin{table}[htbp]
      \begin{center}
        \begin{minipage}{350pt}
         \caption{Numerical results of Example \ref{example1}, using {scheme} (\ref{mfM1}) with homogeneous Dirichlet boundary condition when $\lambda=1$.}\label{table1}
         \setlength{\tabcolsep}{1mm}\begin{tabular}{@{}ccccc|ccccc@{}}
            \hline
            ndof(G1) & $\|u-u_h\|$ & Rate & $\|\nabla_h(u-u_h)\|$ & Rate & ndof(G2) & $\|u-u_h\|$ & Rate & $\|\nabla_h(u-u_h)\|$ & Rate \\ \hline
            370      & 1.43E-1     & --   & 2.65                  & --   & 186      & 1.08E-1     & --   & 2.98                  & --   \\
            1378     & 4.16E-2     & 1.78 & 1.32                  & 1.00 & 690      & 3.22E-2     & 1.75 & 1.54                  & 0.95 \\
            5314     & 1.08E-2     & 1.93 & 6.58E-1               & 1.00 & 2658     & 8.39E-3     & 1.94 & 7.75E-1               & 0.99 \\
            20866    & 2.75E-3     & 1.98 & 3.28E-1               & 1.00 & 10434    & 2.11E-3     & 1.98 & 3.87E-1               & 0.99 \\
            82690    & 6.90E-4     & 1.99 & 1.64E-1               & 1.00 & 41346    & 5.30E-4     & 1.99 & 1.94E-1               & 0.99 \\
            \hline
         \end{tabular}
        \end{minipage}
    \end{center}
   \end{table}

   \begin{table}[htbp]
      \begin{center}
        \begin{minipage}{350pt}
         \caption{Numerical results of Example \ref{example1}, using {scheme} (\ref{mfM1}) with homogeneous Dirichlet boundary condition when $\lambda=10^6$.}\label{table2}
         \setlength{\tabcolsep}{1mm}\begin{tabular}{@{}ccccc|ccccc@{}}\hline
             ndof(G1) & $\|u-u_h\|$ & Rate & $\|\nabla_h(u-u_h)\|$ & Rate & ndof(G2) & $\|u-u_h\|$ & Rate & $\|\nabla_h(u-u_h)\|$ & Rate \\ \hline
             370      & 1.42E-1     &      & 2.63                  &      & 186      & 1.15E-1     &      & 2.95                  &      \\
             1378     & 4.16E-2     & 1.77 & 1.31                  & 1.01 & 690      & 3.46E-2     & 1.74 & 1.51                  & 0.96 \\
             5314     & 1.09E-2     & 1.93 & 6.44E-1               & 1.01 & 2658     & 8.96E-3     & 1.95 & 7.61E-1               & 0.96 \\
             20866    & 2.76E-3     & 1.98 & 3.21E-1               & 1.00 & 10434    & 2.24E-3     & 1.99 & 3.80E-1               & 1.00 \\
             82690    & 6.92E-4     & 1.99 & 1.60E-1               & 1.00 & 41346    & 5.59E-4     & 2.00 & 1.90E-1               & 1.00 \\
            \hline
         \end{tabular}
        \end{minipage}
    \end{center}
   \end{table}

   Note that the discrete $H^1$-seminorm is bounded by the energy norm, we measure
   the error in the discrete $H^1$-seminorm $\|\nabla_h(u-u_h)\|$ to validate theoretical analysis.
   The term `ndof' denotes the number of degrees of freedom, and it is approximately equal to twice the number of vertices
   plus the number of edges in the triangular partitions $\mathcal{T}_h$.
   We focus on the errors and convergence rates on structured mesh grid and unstructured mesh grid (see Figure \ref{meshGrid}).
   The {scheme} (\ref{mfM1}) is used to handle the Dirichlet boundary condition, while the {schemes} (\ref{mfM2}), (\ref{mfM2_2})
   and (\ref{mfM2_3}) are used to handle the mixed boundary conditions.
   From Table \ref{table1}-Table \ref{table2}, we find that {scheme} (\ref{mfM1}) for homogeneous displacement boundary condition has optimal convergence rates,
   and {scheme} (\ref{mfM1}) is parameter-robust about $\lambda$.
   Table \ref{table3} shows the numerical results of Equation (\ref{sch:badscheme}), where the convergence rate of $\vecb{L}^2$-norm decreases.
   So the modification of the left-hand side is necessary.

   \begin{table}[htbp]
    \begin{center}
       \begin{minipage}{350pt}
       \caption{Numerical results of Example \ref{example1}, using the scheme (\ref{sch:badscheme}) with mixed boundary condition when $\lambda=1$.}\label{table3}
       \setlength{\tabcolsep}{1mm}\begin{tabular}{@{}ccccc|ccccc@{}}\hline
            ndof(G1) & $\|u-u_h\|$ & Rate & $\|\nabla_h(u-u_h)\|$ & Rate & ndof(G2) & $\|u-u_h\|$ & Rate & $\|\nabla_h(u-u_h)\|$ & Rate \\ \hline
            370      & 1.57E-1     &      & 2.68                  &      & 186      & 1.05E-1     &      & 2.98                  &      \\
            1378     & 4.83E-2     & 1.70 & 1.34                  & 1.00 & 690      & 3.26E-2     & 1.68 & 1.54                  & 0.94 \\
            5314     & 1.37E-2     & 1.81 & 6.70E-1               & 1.00 & 2658     & 1.02E-2     & 1.67 & 7.78E-1               & 0.99 \\
            20866    & 4.06E-3     & 1.76 & 3.39E-1               & 0.98 & 10434    & 3.44E-3     & 1.57 & 3.90E-1               & 0.99 \\
            82690    & 1.32E-3     & 1.62 & 1.74E-1               & 0.98 & 41346    & 1.22E-3     & 1.48 & 1.96E-1               & 0.98 \\
          \hline
       \end{tabular}
    \end{minipage}
    \end{center}
   \end{table}

   From Table \ref{table4}-Table \ref{table7}, all the schemes for mixed boundary conditions have the optimal convergence rates. Especially when $\lambda=10^6$,
   all the schemes are stable and locking-free for nearly incompressible situations.
   Note that the difference between the {schemes} (\ref{mfM2}), (\ref{mfM2_2}) and (\ref{mfM2_3}) does not affect the uniform convergence of $\lambda$.
   For (\ref{mfM2_2}) and (\ref{mfM2_3}), we only take numerical experiments when $\lambda=10^6$.
   The errors of $\vecb{L}^2$-norm and $\vecb{H}^1$-seminorm vary little when $\lambda$ takes different values.

   \begin{table}[htbp]
       \caption{Numerical results of Example \ref{example1}, using the nonsymmetric scheme (\ref{mfM2}) with mixed boundary condition when $\lambda=1$.}\label{table4}
       \setlength{\tabcolsep}{1mm}\begin{tabular}{@{}ccccc|ccccc@{}}\hline
            ndof(G1) & $\|u-u_h\|$ & Rate & $\|\nabla_h(u-u_h)\|$ & Rate & ndof(G2) & $\|u-u_h\|$ & Rate & $\|\nabla_h(u-u_h)\|$ & Rate \\ \hline
            370      & 1.53E-1     &      & 2.66                  &      & 186      & 1.11E-1     &      & 2.99                  &      \\
            1378     & 4.53E-2     & 1.75 & 1.32                  & 1.00 & 690      & 3.23E-2     & 1.79 & 1.54                  & 0.95 \\
            5314     & 1.19E-2     & 1.92 & 6.59E-1               & 1.01 & 2658     & 8.37E-3     & 1.94 & 7.75E-1               & 0.99 \\
            20866    & 3.03E-3     & 1.97 & 3.28E-1               & 1.00 & 10434    & 2.11E-3     & 1.98 & 3.88E-1               & 0.99 \\
            82690    & 7.62E-4     & 1.99 & 1.64E-1               & 1.00 & 41346    & 5.28E-4     & 1.99 & 1.94E-1               & 1.00 \\
          \hline
       \end{tabular}
  \end{table}

  \begin{table}[htbp]
       \caption{Numerical results of Example \ref{example1}, using the nonsymmetric scheme (\ref{mfM2}) with mixed boundary condition when $\lambda=10^6$.}\label{table5}
       \setlength{\tabcolsep}{1mm}\begin{tabular}{@{}ccccc|ccccc@{}}\hline
            ndof(G1) & $\|u-u_h\|$ & Rate & $\|\nabla_h(u-u_h)\|$ & Rate & ndof(G2) & $\|u-u_h\|$ & Rate & $\|\nabla_h(u-u_h)\|$ & Rate \\ \hline
            370      & 1.49E-1     &      & 2.64                  &      & 186      & 1.21E-1     &      & 2.96                  &      \\
            1378     & 4.47E-2     & 1.73 & 1.30                  & 1.02 & 690      & 3.48E-2     & 1.80 & 1.52                  & 0.96 \\
            5314     & 1.18E-2     & 1.91 & 6.45E-1               & 1.01 & 2658     & 9.03E-3     & 1.94 & 7.61E-1               & 0.99 \\
            20866    & 3.01E-3     & 1.97 & 3.21E-1               & 1.00 & 10434    & 2.28E-3     & 1.98 & 3.80E-1               & 1.00 \\
            82690    & 7.52E-4     & 2.00 & 1.60E-1               & 1.00 & 41346    & 5.71E-4     & 1.99 & 1.90E-1               & 1.00 \\
          \hline
       \end{tabular}
 \end{table}

 \begin{table}[htbp]
       \caption{Numerical results of Example \ref{example1}, using the symmetric scheme (\ref{mfM2_2}) with mixed boundary condition when $\lambda=10^6$.}\label{table6}
       \setlength{\tabcolsep}{1mm}\begin{tabular}{@{}ccccc|ccccc@{}}\hline
            ndof(G1) & $\|u-u_h\|$ & Rate & $\|\nabla_h(u-u_h)\|$ & Rate & ndof(G2) & $\|u-u_h\|$ & Rate & $\|\nabla_h(u-u_h)\|$ & Rate \\ \hline
            370      & 1.48E-1     &      & 2.65                  &      & 186      & 1.20E-1     &      & 2.97                  &      \\
            1378     & 4.40E-2     & 1.75 & 1.30                  & 1.01 & 690      & 3.38E-2     & 1.83 & 1.52                  & 0.96 \\
            5314     & 1.16E-2     & 1.91 & 6.46E-1               & 1.01 & 2658     & 8.78E-3     & 1.94 & 7.63E-1               & 1.00 \\
            20866    & 2.96E-3     & 1.97 & 3.21E-1               & 1.00 & 10434    & 2.21E-3     & 1.98 & 3.81E-1               & 1.00 \\
            82690    & 7.45E-4     & 1.99 & 1.60E-1               & 1.00 & 41346    & 5.55E-4     & 1.99 & 1.90E-1               & 1.00 \\
          \hline
       \end{tabular}
 \end{table}

 \begin{table}[htbp]
    \begin{center}
        \begin{minipage}{350pt}
       \caption{Numerical results of Example \ref{example1}, using the nonsymmetric form of {scheme} (\ref{mfM2_3}), with mixed boundary condition when $\lambda=10^6$.}\label{table7}
       \setlength{\tabcolsep}{1mm}\begin{tabular}{@{}ccccc|ccccc@{}}\hline
            ndof(G1) & $\|u-u_h\|$ & Rate & $\|\nabla_h(u-u_h)\|$ & Rate & ndof(G2) & $\|u-u_h\|$ & Rate & $\|\nabla_h(u-u_h)\|$ & Rate \\ \hline
            370      & 1.50E-1     &      & 2.64                  &      & 186      & 1.14E-1     &      & 2.96                  &      \\
            1378     & 4.51E-2     & 1.74 & 1.30                  & 1.01 & 690      & 3.41E-2     & 1.74 & 1.52                  & 0.96 \\
            5314     & 1.19E-2     & 1.91 & 6.45E-1               & 1.01 & 2658     & 8.89E-3     & 1.94 & 7.61E-1               & 0.99 \\
            20866    & 3.03E-3     & 1.97 & 3.21E-1               & 1.00 & 10434    & 2.23E-3     & 1.99 & 3.80E-1               & 1.00 \\
            82690    & 7.62E-4     & 1.99 & 1.60E-1               & 1.00 & 41346    & 5.57E-4     & 2.00 & 1.90E-1               & 1.00 \\
          \hline
       \end{tabular}
    \end{minipage}
    \end{center}
 \end{table}

 Table \ref{table8} and Table \ref{table9} are used to show the gradient-robustness of the {scheme} (\ref{mfM1}) with
 homogeneous displacement boundary condition. From Theorem \ref{ThHomoGradRob},
 we have the bound $$\Vert\vecb{u}_h\Vert_h\leq\frac{c}{\lambda+\mu}\Vert\phi\Vert$$
 for the gradient-robust discretization.
 As a comparison, for non-gradient-robust methods we have the following bound from \cite{basava2022pressure}
 $$
    \Vert\vecb{u}_h\Vert_{1,h} \leq\frac{c}{\mu}\left(\frac{1}{\lambda}+1\right)\Vert\phi\Vert.
 $$
 By analyzing Table \ref{table8} horizontally, we can find that $\Vert\nabla_h\vecb{u}_h\Vert$ is independent of the discretizations.
 And the vertical direction of the table indicates that $\Vert\nabla_h\vecb{u}_h\Vert=\mathcal{O}(\lambda^{-1})$.
 For $\lambda=10^4$ and $\mu\in (0,1]$, $\frac{1}{\lambda+\mu}\approx c$(constant).
 Table \ref{table9} shows that for different scaled $\mu$, the {quantity} $\Vert\nabla_h\vecb{u}_h\Vert$ {only} varies {very} little, which verifies Theorem \ref{ThHomoGradRob}.

 \begin{table}[htbp]
    \begin{center}
        \begin{minipage}{500pt}
       \caption{Norm $\Vert\nabla_h\vecb{u}_h\Vert$ of Example \ref{example2} with $\mu=1$, different $\lambda$ and different discretizations on structured mesh grid (G1).}\label{table8}
       \begin{tabular}{@{}cccccc@{}}\hline
            $\Vert\nabla_h\vecb{u}_h\Vert$ & ndof=370 & ndof=1378 & ndof=5314 & ndof=20866 & ndof=82690 \\ \hline
            $\lambda=1$                    & 1.089E-1 & 1.124E-1  & 1.136E-1  & 1.139E-1   & 1.140E-1   \\
            $\lambda=10$                   & 3.389E-2 & 3.631E-2  & 3.721E-2  & 3.750E-2   & 3.759E-2   \\
            $\lambda=10^2$                 & 4.443E-3 & 4.839E-3  & 4.994E-3  & 5.046E-3   & 5.063E-3   \\
            $\lambda=10^4$                 & 4.608E-5 & 5.032E-5  & 5.198E-5  & 5.256E-5   & 5.274E-5   \\
            $\lambda=10^6$                 & 4.609E-7 & 5.034E-7  & 5.200E-7  & 5.258E-7   & 5.276E-7   \\
          \hline
       \end{tabular}
    \end{minipage}
    \end{center}
 \end{table}

 \begin{table}[htbp]
    \begin{center}
        \begin{minipage}{500pt}
       \caption{Norm $\Vert\nabla_h\vecb{u}_h\Vert$ of Example \ref{example2} with $\lambda=10^4$, different $\mu$ and different discretizations on structured mesh grid (G1).}\label{table9}
       \begin{tabular}{@{}cccccc@{}}\hline
            $\Vert\nabla_h\vecb{u}_h\Vert$ & ndof=370  & ndof=1378 & ndof=5314 & ndof=20866 & ndof=82690 \\ \hline
            $\mu=10^{-6}$                  & 4.6099E-5 & 5.0344E-5 & 5.2009E-5 & 5.2585E-5  & 5.2768E-5  \\
            $\mu=10^{-4}$                  & 4.6099E-5 & 5.0344E-5 & 5.2009E-5 & 5.2585E-5  & 5.2768E-5  \\
            $\mu=10^{-2}$                  & 4.6099E-5 & 5.0344E-5 & 5.2009E-5 & 5.2584E-5  & 5.2768E-5  \\
            $\mu=10^{-1}$                  & 4.6097E-5 & 5.0342E-5 & 5.2007E-5 & 5.2582E-5  & 5.2766E-5  \\
            $\mu=1$                        & 4.6082E-5 & 5.0324E-5 & 5.1987E-5 & 5.2563E-5  & 5.2746E-5  \\
          \hline
       \end{tabular}
    \end{minipage}
    \end{center}
 \end{table}

 \subsection{Cook's Membrane Problem}
 This is a popular benchmark problem \cite{cook1974improved} for linear elasticity. As shown in Figure \ref{figCooksMembrance},
 the domain $\Omega$ is a convex region formed by connecting four vertices (0,0), (48,44), (48,60) and (0,44).
 The displacement boundary condition $\vecb{u}=\vecb{0}$ is imposed on the left side of the domain. A uniform vertical traction is imposed on the right side,
 that is to say, the boundary condition on the right side is $\vecb{g}=(0,\frac{1}{16})^{\top}$. The rest of the boundary has no traction force.
 The body force $\vecb{f}=\vecb{0}$, the elasticity modulus $E=1$, and the Lam\'{e} constants are given by
 $$
    \lambda=\frac{E\nu}{(1+\nu)(1-2\nu)},\quad \mu=\frac{E}{2(1+\nu)}.
 $$
 As $\nu\rightarrow\frac{1}{2}$ and $\lambda\rightarrow\infty$, the material becomes nearly incompressible. We choose the Possion's ratio as $0.33$ and $0.4999$,
 while $\nu=0.33$ denotes copper and $\nu=0.4999$ denotes rubber. There is no analytical solution to this problem.
 We solving this problem using both the classical lagrangian element $\vecb{P}_1$ and the $\vecb{P}_1\oplus \vecb{RT}_0$ element on unstructured triangulation mesh.
 Figure \ref{figCooksP1} and Figure \ref{figCooksP1PlusRT0} show the numerical dilation $\nabla\cdot\vecb{u}_h$
 using the $\vecb{P}_1$ and $\vecb{P}_1\oplus \vecb{RT}_0$ element, respectively.
 When $\nu=0.33$, both numerical methods have good approximation results.
 The area's top-left corner is squeezed and the dilation $\nabla\cdot\vecb{u}_h$ is negative.
 The bottom of the area is stretched and $\nabla\cdot\vecb{u}_h$ is positive.
 When $\nu=0.4999$, the material is nearly incompressible. The classical Galerkin method exhibits locking phenomenon, the dilation oscillation occurs.
 The $\vecb{P}_1\oplus\vecb{RT}_0$ scheme (\ref{mfM2}) yields a good numerical dilation approximation.
 Due to the nearly incompressible feature of the material, the dilation is numerically much smaller than that of the compressible material.

 \begin{figure}
    \centering
    \includegraphics[scale=0.6]{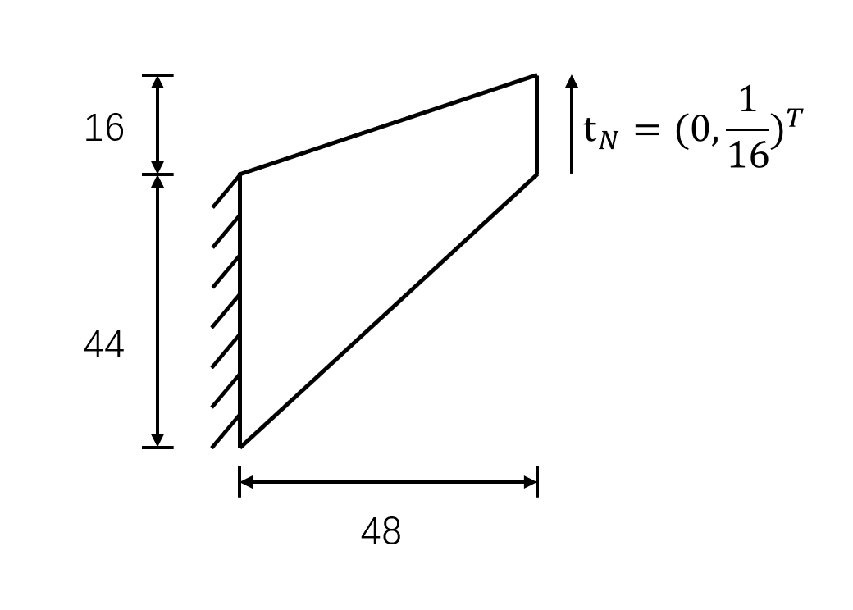}
    \caption{Cook's membrance problem.}
    \label{figCooksMembrance}
 \end{figure}

 \begin{figure}
    \centering
    \includegraphics[scale=0.25]{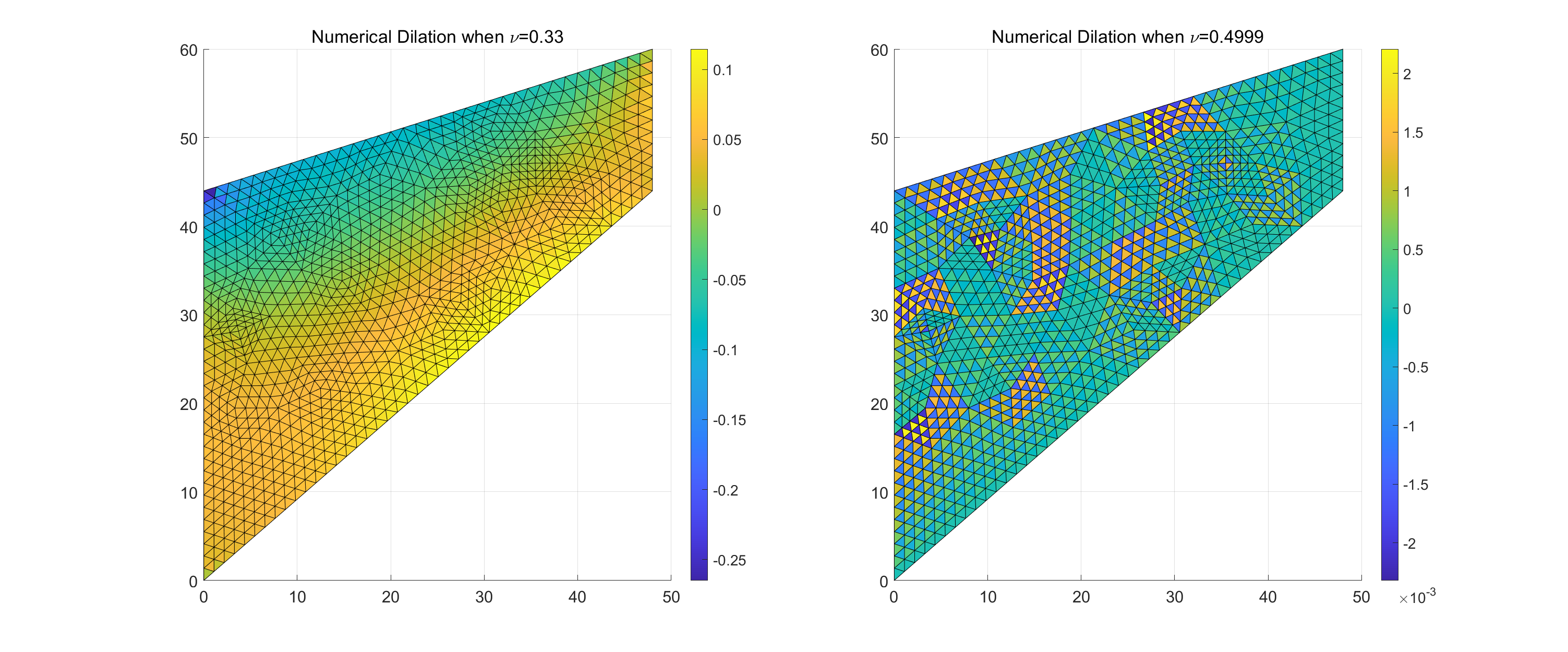}
    \caption{Numerical dilation by $\vecb{P}_1$ element on unstructured mesh. $\nu=0.33$ (left); $\nu=0.4999$ (right).}
    \label{figCooksP1}
 \end{figure}
 \begin{figure}
    \centering
    \includegraphics[scale=0.25]{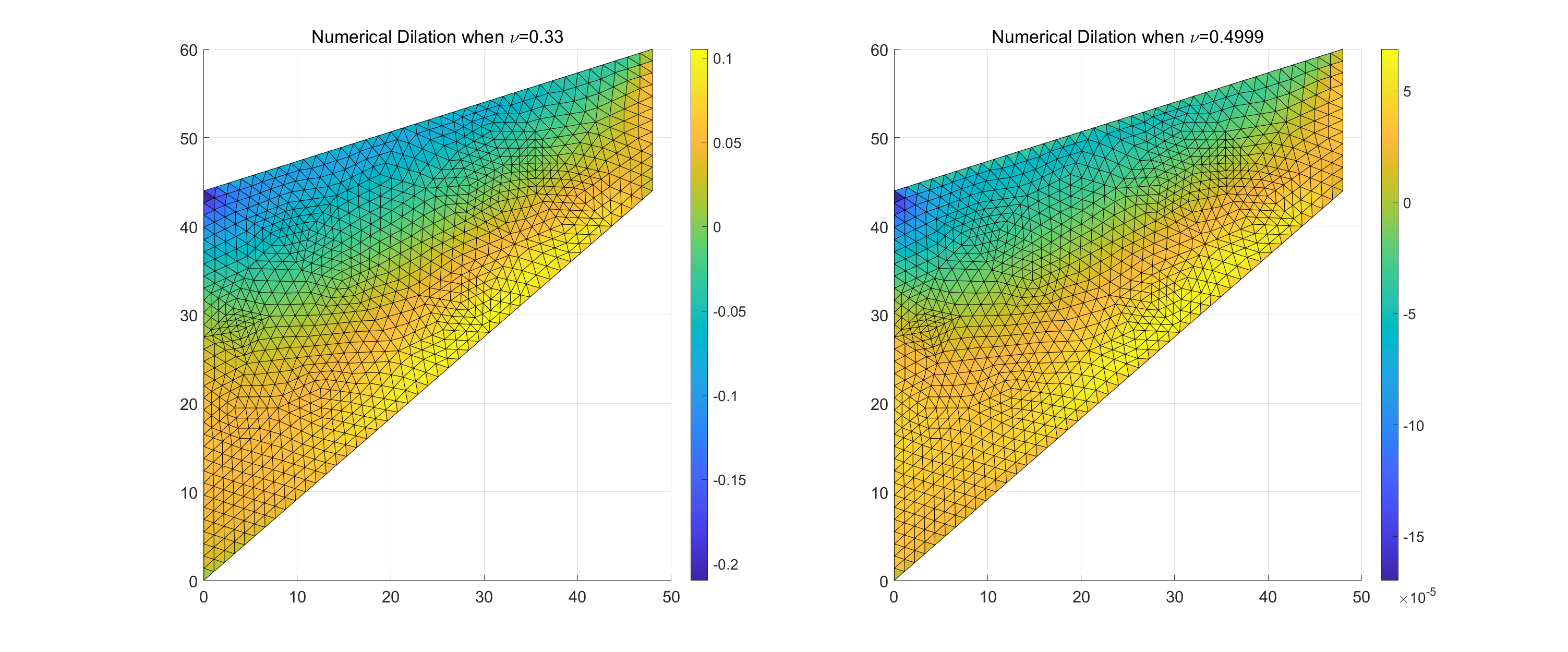}
    \caption{Numerical dilation by $\vecb{P}_1\oplus\vecb{RT}_0$ element on unstructured mesh. $\nu=0.33$ (left); $\nu=0.4999$ (right).}
    \label{figCooksP1PlusRT0}
 \end{figure}

\newpage
\noindent
\textbf{Funding} This work was supported by the National Natural Science Foundation of China (Grant 12131014).\\
\textbf{Data Availability} All data generated or analysed during this study are included in this manuscript.
\section*{Declarations}
\textbf{Conflict of Interest} The authors declare that they have no conflict of interest.

\bibliography{reference}

\end{document}